\documentclass[a4paper, 12pt,]{article}

\usepackage{amsmath,amsthm,amsxtra,amssymb}
\usepackage{graphics}
\usepackage[dvips]{graphicx}
\usepackage{amssymb}
\usepackage{amsmath}
\usepackage{epsfig}
\usepackage{amscd, ascmac}
\usepackage{mathrsfs} 
\usepackage[all]{xy}
\usepackage[dvips]{graphicx}
\usepackage{amssymb, amsmath}
\usepackage{amscd, ascmac}
\usepackage{mathrsfs} 
\usepackage{graphics}
\usepackage[dvips]{graphicx}
\usepackage{eepic}
\usepackage{epic}%

\numberwithin{equation}{section}

\makeindex

\newcommand{\Z}{\mathbb{Z}}

\newcommand{\barr}{\begin{array}}
\newcommand{\earr}{\end{array}}
\newcommand{\beq}{\begin{equation}}
\newcommand{\eeq}{\end{equation}}
\newcommand{\bea}{\begin{eqnarray}}
\newcommand{\eea}{\end{eqnarray}}

\newtheorem{theorem}{\sc theorem}[section]
\newtheorem{proposition}[theorem]{\sc Proposition}

\newtheorem{lemma}[theorem]{\sc Lemma}
\newtheorem{eg}[theorem]{\sc Example}
\newtheorem{definition}[theorem]{\sc Definition}
\newtheorem{remark}[theorem]{\sc Remark}
\newtheorem{corollary}[theorem]{\sc Corollary}

\newcommand{\rcsp}{{\mathbb D}_E}
\newcommand{\rcsptwo}{{{\mathbb D}_{\mathbb{C}}}}
\newcommand{\rcspfour}{{\mathbb D}_{\mathbb{C}^2}}

\newcommand{\prcsp}{\rcsp \times T}

\newcommand{\hexpo}{{\mathcal O}^{\operatorname{exp}}_{\hat{X}}}
\newcommand{\hexpoy}{{\mathcal O}^{\operatorname{exp}}_{\hat{Y}}}

\newcommand{\bexpm}{{\mathcal B}^{\operatorname{exp}}_{\overline{M}}}
\newcommand{\oexp}{{\mathcal O}^{\operatorname{exp}}_{\rcsp}}
\newcommand{\boexp}{{\mathcal B}{\mathcal{O}}^{\operatorname{exp}}_{\overline{N}}}
\newcommand{\pboexp}{\mathscr{H}^{n}_{\partial N}(\hexpo)}

\newcommand{\sspace}{E^{\times}/\mathbb{R}_+}
\newcommand{\wtO}{\widetilde{O}}

\makeatletter
 \def\dedicatory#1{\def\@dedicatory{#1}}
\def\@maketitle{
 \begin{center}
 {\Large\bf Laplace hyperfunctions in several variables \par}
\vspace{5mm}
{\large \@author \par}
\vspace{10mm}
{ \@dedicatory \par}
\end{center}
 \par\vskip 10mm
 }

\makeatother

\title
 \date{}
 \author{Naofumi Honda\,
 \thanks{Department of Mathematics, Hokkaido University, Sapporo 060-0810, Japan,
Supported in part by JSPS Grant-in-Aid for Scientific Research (C) 
(15K04887).
 } 
and\, Kohei Umeta \thanks{Department of Mathematics, Hokkaido University, Sapporo 060-0810, Japan.}}
\dedicatory{ Dedicated to Professor H.~Komatsu on his Sanju \\- eighty years birthday in the traditional Japanese counting.}
 
\begin{document}
\maketitle

\begin{abstract}
We establish  an edge of the wedge theorem for the sheaf of holomorphic functions with exponential growth at infinity and construct the sheaf of Laplace hyperfunctions in several variables.
We also study the fundamental properties  of the sheaf of Laplace hyperfunctions.
\end{abstract}
\section{Introduction}

In the 1980's, H.~Komatsu (\cite{h1}-\cite{h6}) introduced a new class of hyperfunctions in one variable called Laplace hyperfunctions in order to consider the Laplace transform of  a hyperfunction.
By the theory of Laplace hyperfunctions, he gave a new foundation of  Heaviside's theory on a wider class of functions.
A  Laplace hyperfunction in one variable is defined by a difference of boundary values of  holomorphic functions of exponential type along the real axis. Recently, in the paper \cite{hu}, the authors established a vanishing theorem of cohomology 
groups on a Stein open subset
with values in the sheaf of holomorphic functions of exponential type.
As a consequence, we had succeeded in localizing the notion of one dimensional 
Laplace hyperfunctions, that is, 
we constructed the sheaf of Laplace hyperfunctions in one variable (see. \cite{hu}).

\

The aim of our paper is to construct  the sheaf of Laplace hyperfunctions 
in several variables and study its fundamental properties.
For that purpose,  we establish an edge of the wedge type theorem 
for holomorphic functions of exponential type. 
Namely, we show  pure $n$-codimensionality of the partial radial compactification  $\overline{\mathbb{R}^{n}} \times \mathbb{C}^{m}$ of  
 $\mathbb{R}^n \times \mathbb{C}^{m}$ relative
to the sheaf $\hexpo$ of holomorphic functions of exponential type, and
it is the most crucial step for construction of the sheaf $\bexpm$ 
of Laplace hyperfunctions. 
This kind of a theorem, i.e., an edge of the wedge 
theorem for holomorphic functions with bounds,
was first established by T.~Kawai in \cite{k},
where he had shown pure $n$-codimensionality of some compactification of $\mathbb{R}^n$
for the sheaf of holomorphic functions with infra-exponential growth and then constructed
the sheaf of Fourier hyperfunctions. He had effectively used, in his proof, 
a duality theorem between complexes of locally convex topological vector spaces.
After his success, the method employed there becomes very common and develops
in showing pure codimensionality relative to several sheaves of holomorphic 
functions with bounds, see Y.~Saburi (\cite{s}).

The space of holomorphic functions of exponential type appearing in our study
has also a locally convex topology of a vector space, which is defined by
a projective limit of a sequence of dual Fr\'echet spaces. The morphisms
in the sequence are, however, neither homomorphisms onto their images nor compact,
and hence, the resulting topology of a vector space becomes very complicated
like the one for the space of real analytic functions 
on an open subset in $\mathbb{R}^n$. Therefore we cannot apply the topological method
mentioned above directly to the case studied in our paper. To overcome this difficulty,
we adopt an algebraic method based on a Martineau type theorem
(Theorem \ref{martineau}) and obtain the result after some algebraic computations. 
At the same time, we also show 
pure $n$-codimensionality of the boundary 
$\partial\mathbb{R}^{n} \times \mathbb{C}^{m}
:= (\overline{\mathbb{R}^{n}}\setminus \mathbb{R}^n) \times \mathbb{C}^{m}$ 
of the radial compactification relative to $\hexpo$, from
which two important properties of $\bexpm$ follow:
Softness of the sheaf $\bexpm$ and extendability of a usual hyperfunction to
a Laplace hyperfunction.

\

The plan of the paper is as follows. 
In Section $2$, we shortly review the vanishing theorem 
of cohomology groups on a Stein open subset
with values in $\hexpo$,  
which  was given in our previous paper \cite{hu}.
We first recall, for an $n$-dimensional $\mathbb{C}$-vector space $E$ and
a complex linear space $T$ of holomorphic parameters, 
the definitions of the partial radial  compactification $\hat{X}=\rcsp \times T$ of 
$X= E \times T$ and the sheaf $\hexpo$ of holomorphic functions of exponential type on $\hat{X}$.
Since pseudo-convexity is insufficient to guarantee vanishing of higher cohomology groups as Example \ref{eg} shows, to obtain the vanishing theorem, we introduce the notion of the regularity condition at infinity for an open subset in $\hat{X}$.
Then we give the vanishing theorem of cohomology groups and some related results.

The main purpose of Section $3$ is 
to establish an edge of the wedge theorem for $\hexpo$.
Before showing the main theorem, we prepare  several  vanishing theorems of cohomology groups with values in $\hexpo$.
We first show a Martineau type theorem for $\hexpo$ 
which is a key of the proof for the edge of the wedge theorem. 
Then, for an $n$-dimensional $\mathbb{R}$-vector space $M$, 
we show the edge of the wedge theorem along $\overline{M} \times T$ 
for $\hexpo$ and pure
$n$-codimensionality of the boundary $\partial M \times T$ relative to  $\hexpo$.

In Section $4$, we define the sheaf of Laplace hyperfunctions with holomorphic parameters and study its fundamental properties, especially,
we show that any hyperfunction can be extended to a Laplace hyperfunction. 
We also give  the canonical embedding from the sheaf of real analytic 
functions of exponential type to the one of Laplace hyperfunctions.
Softness of the sheaf of Laplace hyperfunctions is shown
in the last section.

\

At the end of the introduction,  
the authors are  grateful to Professor Hikosaburo Komatsu for the valuable lectures and advises.

\section{A vanishing theorem of cohomology groups on a Stein open subset}\label{sec:vanishing_on_stein}

The aim of this section is to review 
the vanishing theorem of cohomology groups 
on a Stein open subset with coefficients in holomorphic functions 
of exponential type.
For the details we refer the reader to  \cite{hu}.

\

Let $n \in \mathbb{N}$ and $E$ be an $n$-dimensional $\mathbb{C}$-vector space with a
norm $|x|$ ($x \in E$).  
We first introduce the radial compactification $\rcsp$ of $E$ and 
the sheaf $\hexpo$ of holomorphic functions of exponential type.
We denote by $E^\times$ the set $E \setminus \{0\}$ and by $\mathbb{R}_+$ the set of
positive real numbers. 
\begin{definition}
The radial compactification $\rcsp$ of $E$ is defined by 
the disjoint union of $E$ and the copy $(\sspace)\infty$ of the quotient space $\sspace$.
The $\rcsp$ is equipped with the topology for which
a sequence $\{x_k\}_{k \in \mathbb{N}}$ of points in $E$ converges to a point 
$x^* \infty \in (\sspace)\infty$ if and only if
\begin{equation}
|x_k| \to \infty\,\,\text{ and }\,\, 
\pi_{E^\times}(x_k) \to x^* \,\, \text{ in $\sspace$}.
\end{equation}
Here $\pi_{E^\times}: E^\times \to \sspace$ is the canonical projection.
\end{definition}
Note that it follows from the definition of $\rcsp$ that 
any linear mapping on $E$ induces the one on $\rcsp$.

Let  $S^{2n-1} $ be the real $(2n-1)$-dimensional unit sphere.
Then the quotient space $\sspace$ can be identified with $S^{2n-1}$ and
this fact is often used in subsequent arguments.
Hence the radial compactification $\rcsp$ of $E$ is identified with the disjoint 
union of $\mathbb{C}^{n}$ and the copy $S^{2n-1}\infty$ of $S^{2n-1}$,
and the topology of $\rcsp$ is given as follows in this identification$:$
Let $D$ be a closed unit ball of center the origin in ${\mathbb C}^n$ 
which is considered as a real $2n$-dimensional topological manifold with the
boundary $S^{2n-1}$.
We define a bijection $\phi(z)$ from $D$ to $\rcsp$ by
$$
\phi(z) =
\left\{
\begin{array}{ll}
\dfrac{z}{1 - \vert z \vert} \in {\mathbb C}^{n}
&\quad (z \in D^\circ), \\
\\
z\infty \in S^{2n-1}\infty  &\quad (z \in \partial D).
\end{array}
\right
.$$
Here $D^\circ$ and $\partial D$ denote the interior and the boundary of $D$ in $\mathbb{C}^{n}$, respectively.
Then $\rcsp$ is equipped with the topology so that $\phi$ gives a topological
isomorphism.

Let $T$ be the linear complex space $\mathbb{C}^m$ ($m \ge 0$)
of holomorphic parameters, and we set $X := E \times T$.
We denote by $\hat{X}$  the partial radial compactification $\prcsp$ of $X$, 
and we also denote by $X_\infty$ the closed subset $\hat{X}\setminus X$ in $\hat{X}$.
Let $(z,\,t)$ (resp. $(z\infty,\,t)$) be a system of coordinates of $X$ 
(resp. $X_\infty$).
A family of fundamental neighborhoods of a point in $\hat{X}$ is given by the following sets;
for a point $(z_0,\, t_0) \in X \subset \hat{X}$,
it is given by
\begin{equation}\label{rcsp-fn-type1}
B_\epsilon(z_0, t_0) := \{(z,t) \in X;\, \vert z - z_0 \vert < \epsilon,\, \vert t - t_0 \vert < \epsilon\}
\end{equation}
for $\epsilon > 0$. On the other hand, for a pint $(z_{0}\infty,\, t_0) \in X_\infty \subset \hat{X} $,  it is
\begin{equation}\label{n.b.h.d}
G_r(\Gamma,\, t_0) := \big(\left\{z \in E;\, \vert z \vert > r,\,\,
\pi_{E^\times}(z) \in \Gamma \right\}
\,\sqcup\, \Gamma\infty \big)\times 
\left\{t \in T;\, \vert t - t_0 \vert < r^{-1}\right\},
\end{equation}
where $r > 0$ and $\Gamma$ runs through open neighborhoods of $z_{0}$ in 
$S^{2n-1}=\sspace$.

Let ${\mathcal O}_{X}$ be the sheaf of holomorphic functions on $X$.
We now define the sheaf of holomorphic  functions of exponential type on $\hat{X}$.

\begin{definition}
Let $\Omega$ be an open subset in $\hat{X}$.
We define the set $\hexpo(\Omega)$ of holomorphic functions of exponential type on $\Omega$ to be the set of all holomorphic functions $f(z,\,t)$ on $\Omega \cap X$ such that,  for any compact set $K$ in $\Omega$, $f(z,\,t)$ satisfies the exponential growth condition
\begin{equation}\label{gc}
|f(z,\,t)|\leq C_{K} e^{H_{K} |z|}  \qquad  ((z,\,t) \in K \cap X)
\end{equation}
with positive constants $C_{K}$ and $H_{K}$.
Let us denote by $\hexpo$ the associated sheaf  on $\hat{X}$ of the presheaf $\{\hexpo(\Omega)\}_{\Omega}$.
\end{definition}
Note that the growth condition of $f(z,\,t)$ is imposed only on the variables $z$.
It is easily seen that  the restriction of the sheaf $\hexpo$ to $X$ coincides with the sheaf $\mathcal{O}_{X}$.

We recall the regularity condition at $\infty$ for an open subset in $\hat{X}$
which plays an essential role in showing our vanishing theorem of cohomology groups on a 
Stein open subset for $\hexpo$.

\begin{definition}
Let $A$ be a subset in $\hat{X}$.
A point $(z{\infty},\,t)$ in $X_\infty$ belongs to the set 
$\operatorname{clos}^1_\infty(A) \subset X_\infty$ 
if and only if there exist points $\{(z_k,\, t_k)\}_{k \in \mathbb N}$ in $A \cap X$ that satisfy the following two conditions when $k \to \infty$.
\begin{enumerate}
\item [1.] $(z_k,\, t_k) \to (z{\infty},\,t)$ in $\hat{X}$.
\item [2.] ${|z_{k+1}|} /{|z_{k}|} \rightarrow 1$.
\end{enumerate}
\end{definition}
Note that the above definition
is independent of the choice of a system of coordinates of $\hat{X}$ and a norm on $E$.
Define 
\begin{equation}
N^1_\infty(A) := X_\infty \setminus 
\operatorname{clos}^1_\infty(X \setminus A).
\end{equation}
We can confirm that, if $U$ is an open subset in $\hat{X}$, 
then $N^{1}_{\infty}(U) \supset U\cap X_{\infty}$ holds.

\begin{definition}
An open subset $U \subset \hat{X}$ satisfying $N^1_\infty(U) = U \cap X_\infty$ is said to be regular at $\infty$. 
\end{definition}
It is easily seen that a finite intersection of open subsets which are regular at
$\infty$ is again regular at $\infty$.
We also give a sufficient condition for which an open subset becomes regular at $\infty$.
Let $A \subset \hat{X}$ be a subset. Set
\begin{equation}
N^L_\infty(A) :=
\left\{(\zeta \infty,\, t) \in X_\infty;\,
(\zeta \infty,\, t) \in \overline{({\mathbb R}_+ \zeta \times \{t\})\, \cap\, A} 
\right\} \subset X_\infty.
\end{equation}
Here ${\mathbb R}_+\zeta$ is the real half line $\pi_{E^\times}^{-1}(\zeta)$ in $E$ 
and the closure $ \overline{({\mathbb R}_+ \zeta \times \{t\})\, \cap\, A} $ is taken in $\hat{X}$. 
Then we have the following lemma.

\begin{lemma}[\cite{hu}]
Let $U\subset \hat{X}$ be an open subset. 
If $N^L_\infty(U) = U \cap X_\infty$ holds, then $U$ is regular at
$\infty$.
\end{lemma}
Note that a finite union of open subsets which satisfy the condition given in the above
lemma is also regular at $\infty$.
We give some examples of open subsets which are regular at $\infty$:
Let $\rcsptwo$ denote the radial compactification of $\mathbb{C}$, and let $\overline{\mathbb{R}}$
be the closure of $\mathbb{R}$ in $\rcsptwo$. Note that $\overline{\mathbb{R}}$ consists
of $\mathbb{R}$ and two points $\{\pm \infty\}$.

\begin{eg} [\cite{hu}, Example 3.6]{\normalfont
Let $U$ be the open set $G_r(\Gamma, 0) \cup \tilde{U}$ where
$\tilde{U}$ is a bounded open subset in $X$ and the cone 
$G_r(\Gamma,\, 0)$ was defined
by $(\ref{n.b.h.d})$ with $r > 0$ and $\Gamma$ being an open subset in $S^{2n-1}=\sspace$.
Then $U$ is regular at $\infty$ as we have $N^L_\infty(U) = U \cap X_\infty$. 
In particular,
${\rcsptwo}$ and $\rcsptwo \setminus [a, +\infty]$ 
$(a \in [-\infty, \infty))$ are regular at $\infty$.
}
\end{eg}
\begin{eg}[\cite{hu}, Example 3.6]{\normalfont
For the set $U:=\rcsptwo\setminus\{1,2,3,4, \dots, +\infty\}$,
we have $N^1_\infty(U) = S^1 \infty \setminus \{+\infty\}$. Hence $U$ is
regular at $\infty$.
However, for the set $U:=\rcsptwo\setminus\{1,2,4,8,16,\dots, +\infty\}$, 
$U$ is not regular at $\infty$ because of $N^1_\infty(U) = S^1 \infty$. 
Note that we have $N^L_\infty(U) = S^1\infty$ for
both the cases.}
\end{eg}

We prepare some notations before stating the theorem.
For a subset $A \subset X$, we set
\begin{equation}
\operatorname{dist}(p, A) =
\left\{
\begin{array}{ll}
\inf_{q \in A} \vert p - q \vert
&\qquad \mathrm{if} \,\,\, A \neq \emptyset, \\
\\
+\infty  &\qquad \mathrm{if}\,\,\,  A = \emptyset.
\end{array}
\right.
\end{equation}
Let $p_{2} : {X=E \times T} \rightarrow T$ be the canonical projection.
We also set, for $q = (z,t) \in X$,
\begin{equation}
\operatorname{dist}_{E}(q,\, A) := 
\operatorname{dist}(q,\, A \cap p_2^{-1}(p_2(q)))
= \underset{(\zeta,\, t) \in A}{\operatorname{inf}}|z - \zeta|.
\end{equation}
For an open subset $\Omega \subset \hat{X}$, we define the function 
\begin{equation}{\label{eq:def-psi}}
\begin{aligned}
\psi(p) &:= 
\min\left\{\dfrac{1}{2},\,\, \dfrac{\operatorname{dist}_{E}( p,\, X \setminus
\Omega)}
                        {1 + \vert z \vert}\right\}\quad
\mathrm{for} \,\,\,p = (z,t) \in X \\
\end{aligned}
\end{equation}
and we  put, for $\epsilon>0$,
\begin{equation}
\begin{aligned}
\Omega_{\epsilon} &:= \left\{p=(z,t) \in \Omega \cap X;\,
\operatorname{dist}(p,\, X \setminus \Omega) > \epsilon,\,
|t| < \dfrac{1}{\epsilon}\right\}.
\end{aligned}
\end{equation}
Note that the function $\psi(p)$ is lower semicontinuous and
continuous with respect to the variables $z$, however,
it is not necessarily continuous with respect to the variables $t$.
\begin{theorem}[\cite{hu},\, Theorem 3.7]{\label{cvs}}
Let $\Omega$ be an open subset in $\hat{X}$.
Assume the following conditions 1.~and 2.
\begin{enumerate}
\item [1.] $\Omega \cap X$ is pseudo-convex in $X$ and
$\Omega$ is regular at $\infty$. 
\item [2.] At a point in $\Omega \cap X$ sufficiently close to $z=\infty$ the 
$\psi(z,t)$ is continuous and
uniformly continuous with respect to the variables $t$, that is,
for any $\epsilon > 0$, there exist $\delta_\epsilon > 0$ and $R_\epsilon > 0$ for
which
$\psi(z,t)$ is continuous on 
$\Omega_{\epsilon,\, R_\epsilon} := \Omega_\epsilon \cap \{|z| > R_\epsilon\}$ 
and it satisfies
$$
\left|\psi(z,\,t) - \psi(z,\, t')\right| < \epsilon\,\qquad
((z,\,t),\, (z,\,t') \in \Omega_{\epsilon,\, R_\epsilon},\,\,
\vert t - t' \vert < \delta_\epsilon).
$$
\end{enumerate}
Then we have
\begin{equation}\label{th:main-result}
\operatorname{H}^k(\Omega,\, \hexpo) = 0 \qquad (k \ne 0).
\end{equation}
\end{theorem}
.

\begin{remark}[\cite{hu}, \,Corollary 3.8]\label{remark}{\normalfont
 Let $U\times V$ be an open subset of product type in 
 $\hat{X} = \rcsp\times T$.
Assume that $U$ satisfies the regularity condition at $\infty$ and that 
$(U \times V) \cap X$ is a pseudo-convex open subset 
in $X$. Then,
since an open subset $U\times V$ of product type always satisfies 
the condition 2.~in  the theorem,  we have}
\begin{equation}
\mathrm{H}^k(U\times V,\, \hexpo) = 0 \qquad  ( k \ne 0).
\end{equation}
\end{remark}

Furthermore, when $n=1$, we can obtain a much stronger result. Remember that
$\rcsptwo$ denotes the radial compactification of $\mathbb{C}$.
\begin{theorem}[\cite{hu}, \,Theorem 5.2]\label{dim-1}
Let $U$ be an open subset in $\rcsptwo$, and let $W$ be a pseudo-convex open subset in $T$. Then we have
\begin{equation}
\operatorname{H}^{k}(U \times W,\,\mathcal{O}^{\exp}_{\hat{X}}) =0 \qquad(k\neq0).
\end{equation}
\end{theorem}
 However, if $n \geq 1$,
the theorem does not hold without the regularity condition as the example  below shows.

\begin{eg}[\cite{hu},\,Example 3.17]\label{eg}
{\normalfont
We consider the case $n=2,\,m=0$, i.e., $X = E = {\mathbb C}^2$ and $\hat{X} = \rcsp =
\rcspfour$.
Set
$$
\begin{aligned}
U&:=\left\{(z_1,\, z_2) \in X;\, |\arg(z_1)| < \dfrac{\pi}{4},\, |z_2| <
|z_1|\right\},\\
\Omega&:=\Big( \overline{U} \Big)^\circ \setminus \{(1,0)\infty\} \subset \hat{X}.
\end{aligned}
$$
Here $(1,0)\infty$ is the point in $(E^\times/\mathbb{R}_+) \infty$ which is the image of
the point $(1,\,0) \in \mathbb{C}^2$ by the canonical projection $\pi_{E^\times}$.
It is easy to check  that $\Omega \cap X = U$ is pseudo-convex in $X$ and $\Omega$ is not regular
at $\infty$.
In this case, we have $\mathrm{H}^1(\Omega,\, \hexpo) \ne 0$.
Therefore  the vanishing theorem does not hold without the regularity condition.
}
\end{eg}

\section{An edge of the wedge theorem for the sheaf $\mathcal{O}^{\exp}_{\hat{X}}$}

The purpose of this section is to show an edge of the wedge theorem 
for $\mathcal{O}^{\exp}_{\hat{X}}$, which is done in Subsection 3.2.
We first prepare a Martineau type theorem for $\hexpoy$.
\subsection{A Martineau type theorem for the sheaf $\hexpoy$}

Let $m \ge 0$ and $n \ge 1$.
Set $T := \mathbb{C}^m$ and
$$
Y :=\mathbb{C} \times \mathbb{C}^{n-1} \times T\,\, \subset \,\,\,
\hat{Y} := \rcsptwo \times \mathbb{C}^{n-1} \times T,
$$
where $\rcsptwo$ denotes the radial compactification of $\mathbb{C}$.
We denote by $\overline{\mathbb{R}}$  the closure of $\mathbb{R}$ in $\rcsptwo$,
which is $\mathbb{R} \sqcup \{\pm \infty\}$.

\begin{theorem}[A Martineau type theorem for the sheaf $\hexpoy$]\label{martineau}
Let $S=[a,\,+\infty] \,\,(a\in\mathbb{R} \cup \{+\infty\})$ be a compact subset in $\overline{\mathbb{R}}$, and let $K=K_{1}\times \cdots \times K_{n-1} \subset L=L_1\times \cdots \times L_{n-1}$ be a pair of closed polydiscs in $\mathbb{C}^{n-1}$. Assume that
$W \subset V$ are non-empty connected Stein open subsets in $T$.
Then the restriction morphism
\begin{equation}\label{martinue}
\mathrm{H}^{n}_{S\times K\times V} (\rcsptwo \times \mathbb{C}^{n-1}\times V,\,\hexpoy) \rightarrow \mathrm{H}^{n}_{ S\times L\times W } (\rcsptwo \times \mathbb{C}^{n-1} \times W,\,\hexpoy)
\end{equation}
is injective.
\end{theorem}
\begin{proof}
Let $(z,\,w,\,t)$ be the coordinates of $Y=\mathbb{C}_{z} \times \mathbb{C}^{n-1}_{w} \times T_{t}$.
We first consider the representation of $\mathrm{H}^{n}_{S\times K\times V} (\rcsptwo \times \mathbb{C}^{n-1}\times V,\,\hexpoy)$ by a relative Leray covering for the sheaf $\hexpoy$.
Choose a point $b<a$ and a sufficiently small constant $\theta>0$. Set
\begin{equation}
\Sigma := \{z \in \mathbb{C} ;\, |\mathrm{arg}(z-b)| < \theta  \} 
\sqcup \{e^{\sqrt{-1} \tau } \infty \in S^1 \infty;\, |\tau| < \theta\} \subset \rcsptwo.
\end{equation}
We also set
\begin{equation}
\begin{aligned}
&U := \Sigma \times \mathbb{C}^{n-1}_{w} \times V,\\
&U_{0} :=(\Sigma \setminus S) \times \mathbb{C}^{n-1}_w \times V,\\
&U_{j}  := \Sigma \times \pi_j^{-1}\left(\mathbb{C}_{w_j} \setminus K_j\right) \times V  
\quad (j= 1,\,\dots, \,n-1),
\end{aligned}
\end{equation}
where $\pi_j: \mathbb{C}_w^{n-1} \to \mathbb{C}_{w_j}$ is the $j$-th projection
of $\mathbb{C}^{n-1}_w$, i.e., 
$\pi_j(w_1,\dots,w_j,\dots, w_{n-1}) = w_j$.
Then  $\mathcal{U}=\{ U,\,U_{0},\,\dots,\,U_{n-1}\}$ and $\mathcal{U}^{'}=\{U_{0},\,\dots,\,U_{n-1}\}$ give  a relative open covering  of the pair
$(U,\,U\setminus (S\times K \times V))$.
On account of Theorem  $\ref{dim-1}$, 
these open sets form  a Leray covering of the pair for the sheaf $\hexpoy$. 
Therefore we obtain the following representation of cohomology groups by the relative $\check{\mathrm{C}}$ech cohomology of $(U,\,U\setminus (S\times K \times V))$ with coefficients in $\hexpoy$.

\begin{equation}
\mathrm{H}^{n}_{S\times K\times V}(\rcsptwo \times \mathbb{C}^{n-1}\times V,\,\hexpoy) =\mathrm{H}^{n}(\mathcal{U}\, \mathrm{mod}\, \mathcal{U^{'}},\,\hexpoy)  =  \frac{\hexpoy(U_{S,\,K,\,V})}{\overset{n-1}{\underset{j =0}{\bigoplus}}\,\, \hexpoy(U^{(j)}_{S,\,K,\,V})}.
\end{equation}
Here $U_{S,\,K,\,V}$ and $U^{(j)}_{S,\,K,\,V}$ ($j=0,\,1,\,\dots,\, n-1$) are defined by
\begin{equation}\label{set}
U_{S,\,K,\,V}:= \underset{0 \le k \le n-1}{\bigcap} U_k\,\,
\text{ and }\,\,
U^{(j)}_{S,\,K,\,V} := \underset{0 \le k \le n-1,\, k \ne j}{\bigcap} U_k
\quad  (j=0,\,\dots,\,n-1).
\end{equation}
Similarly, we define $U_{S,\,L,\,W}$ and $U_{S,\,L,\,W}^{(j)}$ 
by replacing $K$ and $V$ with $L$ and  $W$, respectively.
Thus, the theorem follows from injectivity of the following canonical morphism $\iota$
associated with the restriction map of a holomorphic function.
\begin{equation}
\iota\, : \,  \frac{\,\hexpoy(U_{S,\,K,\,V})}{\overset{n-1}{\underset{j =0}{\bigoplus}}\,\, \hexpoy(U^{(j)}_{S,\,K,\,V})} \longrightarrow 
\frac{\hexpoy(U_{S,\,L,\,W})}{\overset{n-1}{\underset{j =0}{\bigoplus}}\,\,\hexpoy(U^{(j)}_{S,\,L,\,W})}.
\end{equation}

Let us show injectivity of $\iota$.
We take an arbitrary point $(z,\,w,\,t)\in U_{S,\,K,\,V}\cap Y$ and a path
$\gamma_{j}\,\, (1\leq j \leq n-1)$ in $\mathbb{C} \setminus K_{j}$ which encircles $K_{j}$ with clockwise direction such that the point $w_{j}$ is outside $\gamma_{j}$ in the $\mathbb{C}$-plane.
For an element $F(z,\,w,\,t) \in \hexpoy(U_{S,\,K,\,V} )$, we define
\begin{equation}
G(z,\,w,\,t) := \frac{1}{(2\pi \sqrt{-1})^{n-1}}\int_{\gamma_{1} \times \dots \times \gamma_{n-1}}\frac{F(z,\,\mu,\,t)}{(\mu_{1}-w_{1}) \cdots (\mu_{n-1} - w_{n-1})}d\mu.
\end{equation}
It is obvious that  $G(z,\,w,\,t)$ becomes a holomorphic function of exponential type on $U_{S,\,K,\,V}$ by deformation of the path of the integration.
We also take the path $\tilde{\gamma_{j}}\, \, (1\leq j \leq n-1)$ in $\mathbb{C} \setminus L_{j}$ which encircles $L_{j}$ with clockwise direction such that the point $w_{j}$ is inside $\tilde{\gamma_{j}}$.
By deformation of the path of the integration  and Cauchy's integral formula, we obtain
\begin{equation}
\begin{aligned}
&G(z,\,w,\,t) = F(z,\,w,\,t) + \sum_{j=1}^{n-1} H_{j}(z,\,w,\,t),\\
&H_{j}(z,\,w,\,t)\\
& := \frac{1}{(2\pi \sqrt{-1})^{n-j}} \int_{\tilde{\gamma_{j}}\times \gamma_{j+1}\times \dots \times \gamma_{n-1}} \frac{F(z,\,w_{1},\,\dots,\,w_{j-1},\,\mu_{j},\,\dots,\mu_{n-1},\,t)}{(\mu_{j}-w_{j})\cdots(\mu_{n-1} -w_{n-1})}d\mu_{j} \dots d\mu_{n-1}.
\end{aligned}
\end{equation}
Note that $H_{j}(z,\,w,\,t)$ becomes a holomorphic function 
of exponential type on $U^{(j)}_{S,\,K,\,V}$ 
by deformation of the path of the integration.
Hence $\sum_{j=1}^{n-1} H_{j}(z,\,w,\,t)$ belongs to ${\bigoplus}^{n-1}_{j=1}\,\, \hexpoy(U^{(j)}_{S,\,K,\,V})$.
If we could prove that $G(z,\,w,\,t)$ can be extended to  $U_{S,\,K,\,V}^{(0)}$ 
as a holomorphic function of exponential type when the restriction of 
$F$ to ${U_{S,\,L,\,W }}$ vanishes in  
$\hexpoy(U_{S,\,L,\,W}) / {\bigoplus}^{n-1}_{j=0}\,\, \hexpoy(U^{(j)}_{S,\,L,\,W})$, 
then we get injectivity of the morphism $\iota$. 

Suppose  $F|_{U_{S,\,L,\,W}} =0$ in $\hexpoy(U_{S,\,L,\,W}) / {\bigoplus}^{n-1}_{j=0}\,\, \hexpoy(U^{(j)}_{S,\,L,\,W})$. Then there exist functions $ \{F_{j}\}_{j} \in {\bigoplus}^{n-1}_{j=0}\,\, \hexpoy(U^{(j)}_{S,\,L,\,W})$ with $F= \sum_{j} F_{j}$ on $U_{S,\,L,\,W}$.
For a point $(z,\,w,\,t)\in U_{S,\,K,\,V}$, we take a path $\Gamma$ in $\Sigma$ which is composed of a ray from $e^{-\sqrt{-1}\alpha}\infty \,(0<\alpha <\theta)$ to a point $c \,(b<c<a)$ and a ray from $c$ to $e^{\sqrt{-1}\alpha}\infty$ such that $z$ 
is outside the region surrounded by $\Gamma$. 
Similarly, we take a path $\Gamma^{'}$ in $\Sigma$ which is composed of a ray from $e^{-\sqrt{-1}\alpha^{'}}\infty \,(\alpha <\alpha^{'} <\theta)$ to a point $c^{'} \,(b<{c}^{'}<c)$ and a ray from $c^{'}$ to $e^{\sqrt{-1}\alpha^{'}}\infty$ such that $z$ is inside
the region surrounded by $\Gamma^{'}$ as shown in Fig.~\ref{fig:path-in-Cz}.
\begin{figure}
\begin{center}
\scalebox{0.3}{\includegraphics[,]{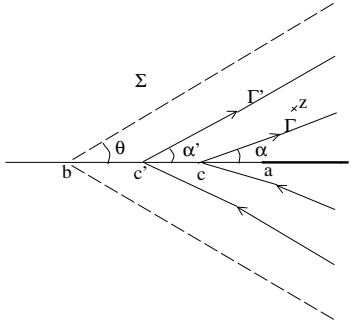}}\\
\caption{The paths in the $\mathbb{C}_z$-plane.}
\label{fig:path-in-Cz}
\end{center}
\end{figure}
Let $\theta_0 > 0$ be a sufficiently small positive real number, and set
$$
\widetilde{Y} := Y \times \{\eta \in \mathbb{C};\, |\arg{\eta}| < \theta_0\}
\subset Y \times \mathbb{C}_\eta.
$$
Define
$$
\begin{aligned}
G_{0}(z,\,w,\,t;\, \eta) &:= \frac{e^{\eta z}}{2\pi \sqrt{-1}} 
\int_{\Gamma }\frac{G(\zeta,\,w,\,t)e^{- \eta \zeta}}{\zeta -z}d\zeta, \\ 
G_{1}(z,\,w,\,t;\, \eta) &:= \frac{e^{\eta z}}{2\pi \sqrt{-1}}\int_{\Gamma^{'} }\frac{G(\zeta,\,w,\,t)e^{-\eta \zeta}}{\zeta -z}d\zeta.
\end{aligned}
$$
Then, since $G(z,\,w,\,t)$ belongs to $\hexpoy(U_{S,\,K,\,V})$, by modifying
the path of the integration, we can find a real-valued continuous function
$\rho_0(z,\,w,\,t)$ (resp. $\rho_1(z,\,w,\,t)$) on $U_{S,\,K,\,V}$
(resp. $U^{(0)}_{S,\,K,\,V}$) for which 
$G_{0}(z,\,w,\,t;\, \eta)$ (resp. $G_{1}(z,\,w,\,t;\, \eta)$)
is holomorphic on $\widetilde{U}_{S,\,K,\,V}$ 
(resp. $\widetilde{U}^{(0)}_{S,\,K,\,V}$), where
$$
\begin{aligned}
&
\widetilde{U}_{S,\,K,\,V} := 
\{(z,\,w,\,t;\,\eta) \in \widetilde{Y};\, (z,\,w,\,t) \in U_{S,\,K,\,V},\,\,
|\eta| > \rho_0(z,\,w,\,t)\}, \\
&
\widetilde{U}^{(0)}_{S,\,K,\,V} := 
\{(z,\,w,\,t;\,\eta) \in \widetilde{Y};\, (z,\,w,\,t) \in U^{(0)}_{S,\,K,\,V},\,\,
|\eta| > \rho_1(z,\,w,\,t)\}.
\end{aligned}
$$
Note that $\rho_0$ (resp. $\rho_1$) is bounded on any compact subset in
$U_{S,\,K,\,V}$ (resp. $U^{(0)}_{S,\,K,\,V}$).
We can also verify that $G_0$ and $G_1$ satisfy the following estimates: 
For any compact subsets $D_0$ in $U_{S,\,K,\,V}$ 
and $D_1$ in $U^{(0)}_{S,\,K,\,V}$, there exist
constants $C_{D_0}$ and $C_{D_1}$ for which we have
$$
\begin{aligned}
&|G_0(z,\,w,\,t;\,\eta)| \le C_{D_0} e^{\operatorname{Re}\eta z}
\qquad (z,\,w,\,t,\,\eta) \in 
\widetilde{U}_{S,\,K,\,V} \cap (D_0 \times \mathbb{C}_\eta),\\
&|G_1(z,\,w,\,t;\,\eta)| \le C_{D_1} e^{\operatorname{Re} \eta z}
\qquad (z,\,w,\,t,\,\eta) \in 
\widetilde{U}^{(0)}_{S,\,K,\,V} \cap (D_1 \times \mathbb{C}_\eta).
\end{aligned}
$$
Furthermore, by Cauchy's integral formula, we have
\begin{equation}\label{cauchy}
G(z,\,w,\,t) = G_{0}(z,\,w,\,t;\,\eta) - G_{1}(z,\,w,\,t;\, \eta).
\end{equation}

Now, under the assumption of $F$, we will show that 
$G_0$ identically vanishes and $G_1$ does not depend on
the variable $\eta$, i.e., $\dfrac{\partial G_1}{\partial \eta} \equiv 0$. We will show only
$\dfrac{\partial G_1}{\partial \eta} \equiv 0$ because $G_0 \equiv 0$ can be proved in the same way. 

Let $w^* \in 
(\mathbb{C}\setminus L_1) \times \dots \times (\mathbb{C}\setminus L_{n-1})$ 
and $t^* \in W$. 
Suppose $(w,\,t)$ to be sufficiently close 
to $(w^*,\,t^*)$. 
Then it follows from the assumption $F|_{U_{S,\,L,\,W}} = \sum_j F_j$ that 
$\dfrac{\partial G_{1}}{ \partial \eta}(z,\,w,\,t;\,\eta)$ is equal to
\begin{equation}
\begin{aligned}\label{G_{0}}
\frac{-e^{\eta z}}{ (2\pi \sqrt{-1})^{n}}
\sum_{j=0}^{n-1}
\int_{\Gamma' } \int_{\gamma_{1} \times \dots \times \gamma_{n-1}}\frac{ F_{j}(\zeta,\,\mu,\,t)e^{- \eta \zeta}\,\,d\mu d\zeta }{(\mu_{1}-w_{1}) \cdots (\mu_{n-1} - w_{n-1})}. \\
\end{aligned}
\end{equation}
Here we first modify the path $\gamma_j$ such that each $\gamma_j$ is in
$\mathbb{C} \setminus L_j$, which is possible because we take $(w,\,t)$ sufficiently
close to $(w^*,\,t^*)$, 
and then, we put $F = \sum_j F_j$ into the definition of $\dfrac{\partial G_1}{\partial \eta}$.
Furthermore we assume $|\eta|$ to be sufficiently large so that 
each integral converges absolutely.

Since, for $j \ge 1$, each $F_{j}(\zeta,\,\mu,\,t)$ 
is holomorphic on $U^{(j)}_{S,\,L,\,W} \cap Y$, we have
\begin{equation}
 \int_{\gamma_{1} \times \dots \times \gamma_{n-1}}\frac{ F_{j}(\zeta,\,\mu,\,t)\,\, d\mu }{(\mu_{1}-w_{1}) \cdots (\mu_{n-1} - w_{n-1})}  =0 \qquad(1\leq j \leq n-1).
\end{equation}
Hence $\dfrac{\partial G_{1}}{\partial \eta}(z,\,w,\,t;\, \eta)$ is equal to
\begin{equation}
\frac{-e^{ \eta z}}{ (2\pi \sqrt{-1})^{n}}  \int_{\gamma_{1} \times \dots \times \gamma_{n-1}}\frac{d\mu}{(\mu_{1}-w_{1}) \cdots (\mu_{n-1} - w_{n-1})} 
\int_{\Gamma'} F_{0}(\zeta,\,\mu,\,t)e^{- \eta \zeta} d\zeta.
\end{equation}
As $F_0$ is holomorphic on $U^{(0)}_{S,\,L,\,W} \cap Y$ and we can take $|\eta|$ large enough so that $e^{- \eta \zeta}$ compensates the exponential growth of $F_{0}$ at infinity, we have
$
\displaystyle\int_{\Gamma'} F_{0}(\zeta,\,\mu,\,t)e^{- \eta \zeta} d\zeta = 0
$
by rotating the path $\Gamma'$ to the positive real axis in the $\mathbb{C}_\zeta$-plane.
This entails $\dfrac{\partial G_{1}}{\partial \eta}(z,\,w,\,t;\,\eta) = 0$ when 
$(w,\,t)$ is sufficiently close to $(w^*,\,t^*)$ and $|\eta|$ is large enough, 
and thus, 
it follows from
the unique continuation property of a holomorphic function that
$\dfrac{\partial G_1}{\partial \eta}$ identically vanishes on $\widetilde{U}^{(0)}_{S,\,K,\,V}$.

By the same argument as above, we can also get $G_0 \equiv 0$ on 
$\widetilde{U}_{S,\,K,\,V}$,
and hence, we have obtained $G(z,\,w,\,t) = - G_{1}(z,\,w,\,t;\,\eta)$, which implies
that $G(z,\,w,\,t)$ is a holomorphic function of 
exponential type on  $U_{S,\,K,\,V}^{(0)}$. 
This completes the proof.
\end{proof}

We also prepare the lemma below, which is needed 
to show Propositions $\ref{p-lemma}$ and $\ref{cover}$.
\begin{lemma}\label{lemma3.1}
Let $S$ $($resp. $U$$)$ be a closed $($resp. an open$)$ subset in $\rcsptwo$, 
and let $K = K_1 \times \cdots \times K_{n-1} \subset \mathbb{C}^{n-1}$ 
be a product of closed subsets $K_j$ in $\mathbb{C}$ $(j=1,\dots,n-1)$.
Assume that $W$ and $V$ 
are Stein open subsets in $\mathbb{C}^{n-1}$ and $T$, respectively. 
Then we have
\begin{equation}
\mathrm{H}^{k}_{(S \cap U) \times (K \cap W)  \times V} (U \times W \times V,\,\hexpoy) = 0 \qquad (k \geq n+1).
\end{equation}
\end{lemma}
\begin{proof}
Let us consider the representation of 
cohomology groups by a relative Leray covering for the sheaf $\hexpoy$.
Define the canonical projection $\pi_j: \mathbb{C}^{n-1} \to \mathbb{C}$ by
$(w_1,\dots,w_j, \dots, w_{n-1}) \to w_j$, and set
\begin{equation}
\begin{aligned}
&P :=  U \times W \times V ,\\
&P_{0} :=(U \setminus S) \times W \times V,\\
&P_{j}  := U \times (\pi^{-1}_j(\mathbb{C} \setminus K_j) \cap W) \times V
\quad (1 \leq j \leq n-1).
\end{aligned}
\end{equation}
Then  $\{ P,\,P_{0},\,\dots,\,P_{n-1}\}$ and $\{P_{0},\,\dots,\,P_{n-1}\}$ become  
a relative Leray covering of the pair  
$(P,\,P\setminus (S\times (K \cap W) \times V))$ for the sheaf $\hexpoy$ 
by Theorem \ref{dim-1}. 
As the number of open sets in the covering is $n+1$, it is evident that the $k$-th cohomology group vanishes for $k \geq n+1$.
\end{proof}
\begin{remark}
{\normalfont
In the proof of the above lemma, if $S = [a,\,+\infty]$ with $a \in \mathbb{R}$ and 
$U=\rcsptwo$,
then Theorem \ref{cvs} is sufficient to show the lemma because $\rcsptwo\setminus S$ is
regular at infinity in this case. For a general closed subset $S$ such as 
$\{+\infty\}$, however, we really need Theorem \ref{dim-1} which is a
specific feature in the one dimensional case.
}
\end{remark}

As an immediate consequence of the Martineau type theorem for the sheaf $\hexpoy$, 
we have the following proposition. 
\begin{proposition}\label{p-lemma}
Let $S=[a,\,+\infty]$ $(a\in\mathbb{R}\cup \{+\infty\})$ be a compact subset in $\overline{\mathbb{R}}$, and let $K=K_{1}\times \cdots \times K_{n-1}$ be a closed polydisc  in $\mathbb{C}^{n-1}$. Assume that $V$ is a Stein open subset in $T$. Then we have
\begin{equation}\label{eq-lemma}
\mathrm{H}^{k}_{S\times K\times V}(\rcsptwo \times \mathbb{C}^{n-1} \times V,\,\hexpoy) = 0 \qquad  ( k\ne n).
\end{equation}
\end{proposition}
\begin{proof}
For $k \geq n+1$, we obtain the assertion by Lemma \ref{lemma3.1}.
Hence it remains to prove that the $k$-th cohomology group vanishes for $k \leq n-1$.
We use the induction on $n$.
In the case of $n=1$, the claim to be proved is 
\begin{equation}\label{eq-a}
\mathrm{H}^{0}_{S\times V}(\rcsptwo \times V,\,\hexpoy)= \Gamma_{{S\times V}}(\rcsptwo \times V,\,\hexpoy)=0.
\end{equation}
Obviously, we have $(\ref{eq-a})$ by the unique continuation property 
of a holomorphic function.
Suppose that we have, for any cylindrical compact set $K^{'}\subset \mathbb{C}^{n-2}$ and any Stein open subset $V$,
\begin{equation}
\mathrm{H}^{k}_{S\times K^{'}\times V}(\rcsptwo \times \mathbb{C}^{n-2} \times V,\,\hexpoy) = 0 \qquad  ( k \leq n-2).
\end{equation}
Under the assumption, let us show
\begin{equation}\label{mok}
\mathrm{H}^{k}_{S\times K^{'}\times K^{''} \times V}(\rcsptwo \times \mathbb{C}^{n-1} \times V,\,\hexpoy) = 0 \qquad  ( k\leq n-1),
\end{equation}
where $K^{''}$ is an arbitrary compact subset in $\mathbb{C}$.
We consider the following long exact sequence of  cohomology groups.
\begin{equation}\label{seq-1}
\begin{aligned}
&\rightarrow \mathrm{H}^{k}_{S \times K^{'} \times K^{''} \times V}(\rcsptwo \times \mathbb{C}^{n-1}  \times V,\,\hexpoy)
\rightarrow \mathrm{H}^{k}_{S \times {K}^{'}\times \mathbb{C} \times V}(\rcsptwo \times \mathbb{C}^{n-2}\times \mathbb{C}  \times V,\,\hexpoy) \\
&\overset{\iota}{\rightarrow} \mathrm{H}^{k}_{S \times {K}^{'} \times (\mathbb{C} \setminus K^{''}) \times V}(\rcsptwo \times \mathbb{C}^{n-2} \times (\mathbb{C} \setminus K^{''}) \times V,\,\hexpoy) 
\rightarrow. 
\end{aligned}
\end{equation}
It follows from the induction hypothesis that we have
\begin{equation}
\begin{aligned}
&\mathrm{H}^{k}_{S \times {K}^{'}\times \mathbb{C} \times V}(\rcsptwo \times \mathbb{C}^{n-2}\times \mathbb{C} \times V,\,\hexpoy)
=0,\\
&\mathrm{H}^{k}_{S \times {K}^{'} \times (\mathbb{C} \setminus K^{''}) \times V}(\rcsptwo \times \mathbb{C}^{n-2} \times (\mathbb{C} \setminus K^{''}) \times V,\,\hexpoy) =0\quad ( k \leq n-2).
\end{aligned}
\end{equation}
Therefore we obtain $(\ref{mok})$ for $ k \leq n-2$.
On account of the Martineau type theorem for the sheaf $\hexpoy$,  the morphism $\iota$ in $(\ref{seq-1})$ for $k=n-1$ is injective.
Since the exact sequence
\begin{equation}
\begin{aligned}
&0\rightarrow \mathrm{H}^{n-1}_{S \times K^{'} \times K^{''} \times V}(\rcsptwo \times \mathbb{C}^{n-1}  \times V,\,\hexpoy)
\rightarrow \mathrm{H}^{n-1}_{S \times {K}^{'}\times \mathbb{C} \times V}(\rcsptwo \times \mathbb{C}^{n-2}\times \mathbb{C}  \times V,\,\hexpoy)\\
&\overset{\iota}{\rightarrow} \mathrm{H}^{n-1}_{S \times {K}^{'} \times (\mathbb{C} \setminus K^{''}) \times V}(\rcsptwo \times \mathbb{C}^{n-2} \times (\mathbb{C} \setminus K^{''}) \times V,\,\hexpoy)
\end{aligned}
\end{equation}
implies $ \mathrm{H}^{n-1}_{S \times K^{'} \times K^{''} \times V}(\rcsptwo \times \mathbb{C}^{n-1}  \times V,\,\hexpoy)=0$,
the assertion holds for $k\leq n-1$. 
Hence the assertion is true for any $n \geq 1$ by the  induction.
\end{proof}

\begin{corollary}\label{cor}
Let $S=[a,\,+\infty]$ $(a\in\mathbb{R}\cup \{+\infty\})$ be a compact subset in $\overline{\mathbb{R}}$, and let $K=K_{1}\times \cdots \times K_{n-1} \subset L=L_1\times \cdots \times L_{n-1}$ be a pair of closed polydiscs in $\mathbb{C}^{n-1}$. 
Assume that $V$  is a Stein open subset in $T$. 
Then we have
\begin{equation}\label{eq-corollary}
\mathrm{H}^{k}_{S\times (L\setminus K)\times V}(\rcsptwo \times \mathbb{C}^{n-1} \times V,\,\hexpoy) =0 \qquad (k\ne n).
\end{equation}
\end{corollary}
\begin{proof}
We consider the following long exact sequence of  cohomology groups
\begin{equation}\label{39}
\begin{aligned}
&\rightarrow \mathrm{H}^{k}_{S \times K \times V}(\rcsptwo \times \mathbb{C}^{n-1} \times V,\,\hexpoy)
\rightarrow \mathrm{H}^{k}_{S \times L \times V}(\rcsptwo \times \mathbb{C}^{n-1} \times V,\,\hexpoy)\\
&\rightarrow \mathrm{H}^{k}_{S \times (L \setminus K) \times V}(\rcsptwo \times \mathbb{C}^{n-1} \times V,\,\hexpoy)
\rightarrow. 
\end{aligned}
\end{equation}
By Proposition \ref{p-lemma}, we have
\begin{equation}\label{40}
\mathrm{H}^{k}_{S \times K \times V}(\rcsptwo \times \mathbb{C}^{n-1} \times V,\,\hexpoy)
= \mathrm{H}^{k}_{S \times L \times V}(\rcsptwo \times \mathbb{C}^{n-1} \times V,\,\hexpoy) =0 \quad (k\neq n).
\end{equation}
Hence we obtain $($\ref{eq-corollary}$)$ for $k\neq n-1,\,n$. 
Moreover, we have the following exact sequence of  cohomology groups by $(\ref{39})$ and $(\ref{40})$.
\begin{equation}
\begin{aligned}
&0 \rightarrow  \mathrm{H}^{n-1}_{S \times (L \setminus K) \times V}(\rcsptwo \times \mathbb{C}^{n-1} \times V,\,\hexpoy)\rightarrow \mathrm{H}^{n}_{S \times K \times V}(\rcsptwo \times \mathbb{C}^{n-1} \times V,\,\hexpoy)\\
&\overset{\iota}{\rightarrow} \mathrm{H}^{n}_{S \times L \times V}(\rcsptwo \times \mathbb{C}^{n-1} \times V,\,\hexpoy). 
\end{aligned}
\end{equation}
Since the morphism $\iota$ is injective by Theorem \ref{martineau}, we have the corollary.
\end{proof}

Corollary \ref{cor} can be extended to a pair of two analytic polyhedra $K\subset L$.
Let us first recall the definition of an analytic polyhedron. 
\begin{definition}
Let $U$ be a domain in $\mathbb{C}^{n}$.  A compact subset $D$ in $U$ defined by
\begin{equation}
D=\{ z\in U ;\, |F_{1}(z)| \leq 1,\,\dots,\, |F_{N}(z)| \leq 1 \}
\end{equation}
with some finitely many $F_{1},\,\dots,\,F_{N} \in \mathcal{O}_{\mathbb{C}^{n}}(U)$ is called an analytic polyhedron of $U$. 
\end{definition}

\begin{theorem}\label{polyhedra}
Let $S=[a,\,+\infty]$ $(a\in\mathbb{R}\cup \{+\infty\})$ be a compact subset in $\overline{\mathbb{R}}$.
Let $K$ and $L$ be two compact analytic polyhedra in $\mathbb{C}^{n-1}$, and let $V$ be a Stein open subset in $T$. Then
\begin{equation}\label{poly}
\mathrm{H}^{k}_{S\times (L\setminus K)\times V}(\rcsptwo\times \mathbb{C}^{n-1} \times V,\,\hexpoy) =0\qquad (0\leq k \leq n-1). 
\end{equation}
\end{theorem}

Before entering into the proof of Theorem $\ref{polyhedra}$, we recall two well-known lemmas. 
\begin{lemma}[\cite{kaneko}, Corollary\, 5.3.7]\label{inverse} 
Let $X$ be a topological space, and let $S$ be a locally closed subset in $X$.
Let $0 \rightarrow \mathcal{L}_{n} \rightarrow \cdots \rightarrow \mathcal{L}_{1} \rightarrow \mathcal{L}_{0} \rightarrow \mathcal{F} \rightarrow 0$ be an exact sequence of sheaves on $X$.
If
$$
\mathrm{H}^{k}_{S}(X,\,\mathcal{L}_{j}) =0 \qquad (r+j \leq k\leq N+j, \quad j=0,\,\dots,n),
$$
then
$$
\mathrm{H}^{k}_{S}(X,\,\mathcal{F})=0\qquad(r\leq k\leq N).
$$
\end{lemma}

\begin{lemma}[\cite{schapira}, Proposition\,B.4.2]\label{schapiram}
Let $M$ be a module, and let $\phi_{1},\,\dots,\,\phi_{p}$ 
be a family of commuting endomorphisms of $M$. 
Let $\mathcal{M}$ be a Koszul complex associated to the sequence $(\phi_{1},\,\dots,\,\phi_{p})$.
Assume that, for each $1\leq j \leq p$, $\phi_{j}$ is  injective as an endomorphism of the module $\dfrac{M}{\sum^{j-1}_{i=1}\phi_{i}(M)}$.
Then we have
\begin{equation}
 \mathrm{H}^{j}(\mathcal{M}) =
        \left\{
        \begin{array}{ll}
                0 \qquad & (j\neq p), \\
                \\
            \dfrac{M}{ \sum_{i=1}^{p}\phi_{i}(M)}  \qquad & (j=p).
        \end{array}
        \right.
\end{equation}
\end{lemma}

\noindent
\begin{proof}Without loss of generality we can assume that $K$ is contained in $L$.
Hence,  two analytic polyhedra $L$ and $K$ can be expressed as
\begin{equation}
\begin{aligned}
L&=\left\{ w \in \mathbb{C}^{n-1}\,;\, |F_{1} (w) | \leq 1,\cdots,\,|F_{l^{'}} (w) | \leq 1  \right\},\\
K&=\left\{ w \in \mathbb{C}^{n-1}\,;\, |F_{1} (w) | \leq 1, \cdots,|F_{l^{'}} (w) | \leq 1, |F_{l^{'}+1} (w) | \leq 1,\,\cdots, |F_{l} (w) | \leq 1   \right\},
\end{aligned}
\end{equation}
where $F_{1},\,\dots,,\,F_{l}$ are entire functions on $\mathbb{C}^{n-1}$.
We may assume that $L$ is contained in the polydisc of the radius $r$ by the boundedness of $L$. 
Set
\begin{equation}
R:= \max \left\{ 1,\,\, \underset{w\in L}{\sup} |F_{l^{'}+1}(w)|, \cdots, \underset{w\in L}{\sup} |F_{l}(w)|, \,\, r\right\}.
\end{equation}
Let $Z :=\mathbb{C}_z \times \mathbb{C}^{n-1}_{w}\times \mathbb{C}^{l}_{\tilde{w}} \times T_{t}$ and
$\hat{Z} :=\rcsptwo_{z}\times \mathbb{C}^{n-1}_{w}\times \mathbb{C}^{l}_{\tilde{w}} \times T_{t}$.
We consider the closed embedding $\Psi : \hat{Y} \rightarrow \hat{Z}$ defined by 
\begin{equation}
\Psi (z,\,w,\,t) := (z,\, w,\,F_{1}(w),\,\dots,\,F_{l^{'}}(w),\,\dots,\,F_{l}(w),\,t).
\end{equation}
Also set
\begin{equation}
\tilde{L} := \left\{
\begin{array}{clc}
{}& \mid w_{1}\mid \leq R,\,\dots,\, |w_{n-1}| \leq R,\\
(w,\,\tilde{w}) \in \mathbb{C}^{n-1}\times \mathbb{C}^{l}\, ;& \mid \tilde{w}_{1} \mid \leq 1,\,\dots,\, \mid \tilde{w}_{l^{'}} \mid \leq1,  \\
{}& \mid \tilde{w}_{l^{'} +1} \mid \leq R, \dots,\, \mid \tilde{w}_{l} \mid \leq R
\end{array}
\right\}
\end{equation}
and
\begin{equation}
\tilde{K} := \left\{
\begin{array}{clc}
{}& \mid w_{1}\mid \leq R,\,\dots,\, |w_{n-1}| \leq R,\\
(w,\,\tilde{w}) \in  \mathbb{C}^{n-1}\times \mathbb{C}^{l}\, ;& \mid \tilde{w}_{1} \mid \leq 1,\,\dots,\, \mid \tilde{w}_{l^{'}} \mid \leq 1,  \\
{}& \mid \tilde{w}_{l^{'} +1} \mid \leq1, \dots,\, \mid \tilde{w}_{l} \mid\leq 1
\end{array}
\right\}.
\end{equation}
 Since $\Psi$ is a closed immersion and $\Psi^{-1}(S\times(\tilde{L} \setminus \tilde{K})\times V) = S\times(L\setminus K)\times V$ holds, we have
\begin{equation}
\mathrm{H}^{k}_{S\times (L\setminus K)\times V } (\rcsptwo \times \mathbb{C}^{n-1}\times V,\,\hexpoy) \cong
\mathrm{H}^{k}_{S\times (\tilde{L}\setminus \tilde{K}) \times V} (\rcsptwo \times \mathbb{C}^{n-1}\times \mathbb{C}^{l} \times V,\,\Psi _{*}\hexpoy).
\end{equation}
Let $\mathcal{O}^{\exp}_{\hat{Z}}$ denote the sheaf of holomorphic functions of exponential type on $\hat{Z}$.
By Corollary \ref{cor}, we have
\begin{equation}
\mathrm{H}^{k}_{S\times (\tilde{L} \setminus \tilde{K}) \times V}(\rcsptwo \times \mathbb{C}^{n-1}\times \mathbb{C}^{l}\times V ,\,\mathcal{O}^{\exp}_{\hat{Z}}) = 0\qquad (k \leq (n-1) +l).
\end{equation}
Therefore, if we prove that there exists a resolution of 
the sheaf $\Psi_{*} \hexpoy$ by the sheaf $\mathcal{O}^{\exp}_{\hat{Z}}$ of length $l$, 
then the assertion follows from  Lemma \ref{inverse}.
Let us show existence of such a resolution.
By using $F_{j}$, we define a family $\{\phi_{j}\}$ of commuting endomorphisms of  
$\mathcal{O}^{\exp}_{\hat{Z}}$
given by 
\begin{equation}
\phi_{j} (f) := f(z,\,w,\,\tilde{w},\,t) (\tilde{w}_{j} - F_{j}(w))\,\quad (j=1,\,\dots,\,l).
\end{equation}
Let $e_{1},\,\dots,\,e_{l}$ be the canonical basis of $\mathbb{Z}^{l}$.
For an ordered subset $J:=(j_{1},\,\dots,\,j_{k})$ of $\{1,\,\dots,\,l \}$, we set
$e_{J}:=e_{j_{1}} \wedge \dots \wedge e_{j_{k}} \in \overset{k}{\wedge}(\mathbb{Z}^{l})$ and
\begin{equation}
M^{(k)} := \mathcal{O}^{\exp}_{\hat{Z}} \underset{\mathbb{Z}}{\otimes} \overset{k}{\wedge}(\mathbb{Z}^{l}) \quad (k=0,\,1,\,\dots,\,l).
\end{equation}
We also define the differential $d : M^{(k)} \rightarrow M^{(k+1)}$ by 
\begin{equation}
d(f e_{J}) := \sum^{l}_{j=1} \phi_{j}(f)e_{j} \wedge e_{J}, \quad fe_{J} \in M^{(k)}.
\end{equation}
It is easily seen that $d \circ d=0$ from the commutativity of the morphisms $\phi_{j}$'s.
Therefore we  get the following Koszul complex $\mathcal{M}$ associated to the sequence $(\phi_{1},\,\dots,\,\phi_{l})$.
\begin{equation}
\mathcal{M} : \quad 0\rightarrow M^{(0)} \overset{d}{\rightarrow} M^{(1)} \rightarrow \dots  \overset{d}{\rightarrow} M^{(l)} \rightarrow 0.
\end{equation}
Let us see that the complex $\mathcal{M}$ is the desired resolution of the sheaf $\Psi_{*} \hexpoy$.
We prepare the following lemma in order to apply Lemma \ref{schapiram} to the complex $\mathcal{M}$. 

\begin{lemma}\label{shapira}
For each $1\leq k \leq l$ and a point $p \in \hat{Z}$, the morphism $\phi _{k}$ is injective as an endomorphism of the module
$\dfrac{\mathcal{O}^{\exp}_{\hat{Z}\,\,p}}{\sum_{j=1}^{k-1}\phi_{j}(\mathcal{O}^{\exp}_{\hat{Z}\,\,p})}$.
Here $\mathcal{O}^{\exp}_{\hat{Z}\,\,p}$ denotes the stalk of the sheaf $\mathcal{O}^{\exp}_{\hat{Z}}$ at a point $p$.
\end{lemma}
\begin{proof}
Assume that $f \in\mathcal{O}^{\exp}_{\hat{Z}\,\,p}$ satisfies $\phi_{k}(f) = 0$ in $\dfrac{\mathcal{O}^{\exp}_{\hat{Z}\,\,p}}{\sum_{j=1}^{k-1}\phi_{j}(\mathcal{O}^{\exp}_{\hat{Z}\,\,p})}$.
Then there exist $g_{1},\,\dots,\,g_{k-1} \in \mathcal{O}^{\exp}_{\hat{Z}\,\,p}$ such that $\phi_{k}(f) = \sum^{k-1}_{j=1} \phi_{j}(g_{j})$ holds.
Hence we obtain
\begin{equation}\label{sumphi}
f(z,\,w,\,\tilde{w},\,t)(\tilde{w}_{k} -F_{k}(w)) = \sum_{j=1}^{k-1} 
g_{j}(z,\,w,\,\tilde{w},\,t) (\tilde{w}_{j} - F_{j}(w))
\end{equation}
as a relation of germs in $\mathcal{O}^{\exp}_{\hat{Z}\,\,p}$.
By setting $\tilde{w}_{k} = F_{k}(w)$ in $(\ref{sumphi})$, we have
\begin{equation}\label{sumzero}
\sum_{j=1}^{k-1} g_{j}(z,\,w,\,\tilde{w}_{1}, \, \dots,\,\tilde{w}_{k-1},\,
F_{k}(w),\,\tilde{w}_{k+1},\,\dots,\,\tilde{w}_{l},\,t) (\tilde{w}_{j} - F_{j}(w))=0 .
\end{equation}
For $1\leq j \leq k-1$, we set
\begin{equation}
h_{j}(z,\,w,\,\tilde{w},\,t) :=\frac{g_{j}(z,\,w,\,\tilde{w},\,t) - g_{j}(z,\,w,\,\tilde{w}_{1}, \, \dots,\,\tilde{w}_{k-1},\,F_{k}(w),\,\tilde{w}_{k+1},\,\dots,\,\tilde{w}_{l},\,t)}{\tilde{w}_{k} -F_{k}(w)}.
\end{equation}
Then $h_{j}$ belongs to $\mathcal{O}^{\exp}_{\hat{Z}\,\,p}$.
By $(\ref{sumphi})$ and $(\ref{sumzero})$, we have 
\begin{equation}
\begin{aligned}
\sum_{j=1}^{k-1}\phi_{j}(h_{j}(z,\,w,\,\tilde{w},\,t))&= \sum_{j=1}^{k-1}h_{j}(z,\,w,\,\tilde{w},\,t)(\tilde{w_{j}}-F_{j}(w))\\
&=\sum_{j=1}^{k-1}\frac{g_{j}(z,\,w,\,\tilde{w},\,t)}{\tilde{w}_{k} - F_{k}(w)}(\tilde{w}_{j} - F_{j}(w))\\
&=\frac{f(z,\,w,\,\tilde{w},\,t)(\tilde{w}_{k} - F_{k}(w))}{\tilde{w}_{k} - F_{k}(w)}=f(z,\,w,\,\tilde{w},\,t)
\end{aligned}
\end{equation}
as relations of germs in $\mathcal{O}^{\exp}_{\hat{Z}\,\,p}$.
Hence we obtain injectivity of $\phi_{k}$.
\end{proof}
For each point $p\in \hat{Z}$, 
by Lemma \ref{schapiram} and Lemma \ref{shapira}, we have
\begin{equation}\label{coho}
 \mathrm{H}^{k}(\mathcal{M}_{p}) =
        \left\{
        \begin{array}{ll}
                0 \qquad & (k\neq l), \\
                \\
           \dfrac{\mathcal{O}^{\exp}_{\hat{Z}\,\,p}}{\sum_{j=1}^{l}\phi_{j}(\mathcal{O}^{\exp}_{\hat{Z}\,\,p})}  \qquad & (k=l).
        \end{array}
        \right.
\end{equation}
Hence, by Lemma \ref{inverse} and the following lemma, 
we have obtained Theorem $\ref{polyhedra}$. The proof has been completed.
\end{proof}
\begin{lemma}
The following complex of  sheaves on $\hat{Z}$ is exact.
\begin{equation}\label{seq-62}
0 \rightarrow  M^{(0)} \overset{d}{\rightarrow} M^{(1)} \rightarrow \dots  
\overset{d}{\rightarrow} M^{(l)} 
\overset{\rho}{\rightarrow} \Psi_{*}\hexpoy \rightarrow 0.
\end{equation}
Here the sheaf morphism $\rho$ is induced from the one
$\Psi^{-1} \mathcal{O}^{\exp}_{\hat{Z}} \to \mathcal{O}^{\exp}_{\hat{Y}}$
defined by the substituntion $\tilde{w} =(F_{1}(w),\,\dots, F_{l}(w))$. 
\end{lemma}
\begin{proof}
Let us prove exactness of the complex
\begin{equation}\label{stalks}
0 \rightarrow  M^{(0)}_{p} \overset{d}{\rightarrow} M^{(1)}_{p} \rightarrow \dots  \overset{d}{\rightarrow} M^{(l)}_{p} \rightarrow (\Psi_{*}\hexpoy)_{p}\rightarrow 0
\end{equation}
at each point $p\in \hat{Z}$.
Note that we have
\begin{equation}
 (\Psi_{*}\hexpoy)_{p} =
        \left\{
        \begin{array}{ll}
                0 \qquad & (p\notin \Psi(\hat{Y})), \\
                \\
             (\mathcal{O}^{\exp}_{\hat{Y}})_{\Psi^{-1}(p)} \qquad & (p\in  \Psi(\hat{Y})).
        \end{array}
        \right.
\end{equation}
For $p\notin \Psi(\hat{Y})$, it is immediate that the sequence $(\ref{stalks})$ is exact by  $(\ref{coho})$.
Let us show exactness of $(\ref{stalks})$ for $p \in \Psi(\hat{Y})$.
It is sufficient to see that 
$\dfrac{\mathcal{O}^{\exp}_{\hat{Z}\,\,p}}{\sum_{j=1}^{l}\phi_{j}(\mathcal{O}^{\exp}_{\hat{Z}\,\,p})}$ is isomorphic to  
$(\mathcal{O}^{\exp}_{\hat{Y}})_{\Psi^{-1}(p)}$ by $\rho$.
The morphism $\rho$ is clearly surjective  because 
a germ in 
$(\mathcal{O}^{\exp}_{\hat{Y}})_{\Psi^{-1}(p)}$ 
can be regarded as the one in 
${\mathcal{O}^{\exp}_{\hat{Z}\,\,p}}$ 
through the canonical projection
$(z,\,w,\,\tilde{w},\,t) \to (z,\,w,\,t)$.
Furthermore, by considering a Taylor expansion 
with respect to $\tilde{w}_{j}-F_{j}(w)$ $(j=1,\,\dots,\,l)$ at the point $p$, 
we find that its kernel consists of 
the germs in the form $\sum^{l}_{j=1} g_{j}(z,\,w,\,\tilde{w},\,t)(\tilde{w}_{j}-F_{j}(w))$
with $g_j \in \mathcal{O}^{\exp}_{\hat{Z}\,\,p}$. 
Hence we get the exact sequence
\begin{equation}
0\rightarrow \sum_{j=1}^{l}\phi_{j}({\mathcal{O}^{\exp}_{\hat{Z}\,\,p}}) \rightarrow {\mathcal{O}^{\exp}_{\hat{Z}\,\,p}}  \rightarrow
 (\mathcal{O}^{\exp}_{\hat{Y}})_{\Psi^{-1}(p)}\rightarrow 0.
\end{equation}
Therefore we get exactness of $(\ref{stalks})$ on $\hat{Z}$.
\end{proof}

\subsection{Edge of the wedge theorems for the sheaf $\hexpo$}{\label{subsec:edge}}
Using results prepared in previous subsection, we establish two kinds of edge of the wedge theorems for the sheaf $\hexpo$.
Let $M$ be an $n$-dimensional $\mathbb{R}$-vector space ($n \ge 1$) 
with an inner product $(\cdot,\cdot)$,
and let $E$ be its complexification $\mathbb{C} \underset{\mathbb{R}}{\otimes} M$
which is an $n$-dimensional $\mathbb{C}$-vector space 
with the norm induced from the inner product of $M$. 
Then, as in Section \ref{sec:vanishing_on_stein}, we can define the radial compactification $\rcsp$ of $E$. 
We denote by $\overline{M}$ the closure of $M$ in $\rcsp$ and set $\partial M := \overline{M} \setminus M$.
We also set $T := \mathbb{C}^m$ ($m \ge 0)$ and
$$
\begin{array}{ccc}
N := M \times T &\subset &X := E \times T \\
\cap & & \cap \\
\overline{N} := \overline{M} \times T & \subset & \hat{X}:= \rcsp \times T.
\end{array}
$$
Define $\partial N := \overline{N} \setminus N = \partial M \times T$.
Note that $\overline{N}$ is nothing but the closure of $N$ in $\hat{X}$ and
$\partial N$ can be identified with $S^{n-1} \times T$.

\begin{theorem} \label{thm-codimentionality}
The closed subset $\overline{N}$ in $\hat{X}$ is purely n-codimensional relative to the sheaf $\hexpo$, i.e.,
\begin{equation}
\mathscr{H}^{k}_{\overline{N}}(\hexpo)=0 \qquad (k \ne n).
\end{equation}
\end{theorem}
\begin{proof}
Let $p_{\infty}= (1,\, 0,\, \dots,\, 0)\infty \times (0) \in \partial N \subset \hat{X}$.
As it is well-known that $N$ is purely $n$-codimensional relative to the sheaf of holomorphic functions on $X$, the proof is completed by showing 
\begin{equation}
\mathscr{H}^{k}_{\overline{N}}(\hexpo)_{p_{\infty}} = 
0 \qquad (k \neq n).
\end{equation}
For any $\epsilon >0$, we set
\begin{equation}
\begin{aligned}
&U_{\epsilon}  := \left\{ (z_{1}, \dots, z_{n} )\in E; |\arg z_{1}| <\epsilon,\, |z_{1}|>\frac{1}{\epsilon},\,\, |z_{i}|< \epsilon |z_{1}|, \quad i =2, \dots, n  \right\},\\
&T_{\epsilon} :=\{ t\in T; |t| < \epsilon \},\\
&\Omega_{\epsilon} := \overline{U_{\epsilon}}^{\circ} \times T_{\epsilon} \subset \hat{X}.
\end{aligned}
\end{equation}
Here the closure and interior are taken in $\hat{X}$. Then the family $\{ \Omega_{\epsilon} \}_{\epsilon>0}$ 
forms a fundamental system of open neighborhoods of $p_{\infty}$ in $\hat{X}$. Hence, for any $k$, we obtain
\begin{equation}
\mathscr{H}^{k}_{\overline{N}}(\hexpo)_{p_{\infty}} = \underset{\epsilon>0}{\varinjlim}\, \mathrm{H}^{k}_{\overline{N} \cap \Omega_{\epsilon}}( \Omega_{\epsilon} ,\,\hexpo).
\end{equation}
We also set
\begin{equation}
\begin{aligned}
&V_{\epsilon}  :=\left\{ z\in \mathbb{C} ; |\arg z| < \epsilon,\, |z|>\frac{1}{\epsilon} \right\}, \\
&W_{\epsilon} :=\left\{ (w_{1},\, \dots,\, w_{n-1}) \in\mathbb{C}^{n-1} ;  |w_{i}| < \epsilon, \quad i=1, \dots, n-1 \right\},\\
&O_{\epsilon}:={\overline{V_{\epsilon}}}^\circ \times W_{\epsilon} \times T_{\epsilon} \subset \rcsptwo \times \mathbb{C}^{n-1} \times T = \hat{Y}
\end{aligned}
\end{equation}
for any $\epsilon>0$.
Then $\{O_{\epsilon}\}_{\epsilon>0}$ is also a fundamental system of 
neighborhoods of the point 
$q_{\infty} := (1)\infty\times(0)\times (0) \in S^{1}\infty \times 
\mathbb{C}^{n-1}\times T $ in $\hat{Y}$, where $\hat{Y}$ was defined in the previous
subsection.

Let  $\Phi$ be the holomorphic map from $(E \setminus \{z_{1}=0\}) \times T$ 
to $(\mathbb{C}\setminus \{0\})\times\mathbb{C}^{n-1} \times T$ defined by
$$
(z_{1},\, \dots,\, z_{n}) \times t  \,\,\longmapsto \,\, z_{1} \times \left(\frac{z_{2}}{z_{1}},\,\dots,\,\frac{z_{n}}{z_{1}} \right) \times t.
$$
The map $\Phi$ gives a biholomorphic map between $U_{\epsilon} \times T_{\epsilon}$ and $V_{\epsilon} \times W_{\epsilon} \times T_{\epsilon} $.
Let us  regard $S^{2n-1}$ (resp. $S^{1}$) as the boundary of the  
closed unit ball of center the origin
in $E$ (resp. $\mathbb{C}$).
We define the bijection $\phi$ from $(S^{2n-1} \setminus \{z_{1}=0\})\infty \times T$ 
to $S^{1}\infty \times \mathbb{C}^{n-1} \times T$ 
by the following correspondence;
\begin{equation}
\begin{aligned}
&(z_{1},\,z_{2},\,\dots,\,z_{n})\infty \times t \,\,\longmapsto \,\, \left(\frac{z_{1}}{ |z_{1}|} \right)\infty \times \left(\frac{z_{2}}{ z_{1}} ,\,\dots,\,\frac{z_{n}}{ z_{1}}  \right) \times t,\\
&\left(\frac{z}{\sqrt{1+|w|^{2} }},\,\frac{zw}{\sqrt{1+|w|^{2} }} \right)\infty \times t
 \,\,\reflectbox{$\longmapsto$}\,\, z\infty \times w \times t.
\end{aligned}
\end{equation}
\begin{lemma}
	The map 
\begin{equation}
\widehat{\Phi} = \Phi \sqcup \phi: \left(\rcsp\setminus\overline{\{z_1 = 0\}}\right)
	\times T \to 
\left(\rcsptwo \setminus \{0\}\right) 
\times \mathbb{C}^{n-1} \times T
\end{equation}
is a homeomorphism and we have
$$
\widehat{\Phi}(\Omega_{\epsilon}) = O_{\epsilon}, \quad
\widehat{\Phi}\left(\overline{N} \cap \Omega_{\epsilon}\right) =
\left(\overline{\mathbb{R}}\times \mathbb{R}^{n-1} \times T\right) 
\cap O_\epsilon.
$$
Furthermore, for any relatively compact open subset $\Omega$ 
in $\rcsp\setminus\overline{\{z_1 = 0\}}$, the holomorphic map $\Phi$ gives
the sheaf isomorphism on $\Omega \times T$
\begin{equation}
\hexpo\big|_{\Omega\times T} \simeq 
\widehat{\Phi}^{-1}
\left(\mathcal{O}^{\exp}_{\widehat{Y}}\big|_{\widehat{\Phi}(\Omega \times T)}\right).
\end{equation}
\end{lemma}
\begin{proof}
It suffices to show $\widehat{\Phi}$ and $\widehat{\Phi}^{-1}$ 
to be continuous as the rest of claims in
the lemma can be easily confirmed.
Let  $\omega_{i} =(\omega_{i,\,1},\,\dots,\,\omega_{i,\,n}) \,\,(i=1,\,2,\,\dots)$  and     $\omega_{\infty} =(\omega_{\infty,\,1},\,\dots,\,\omega_{\infty,\,n})$ be points in $S^{2n-1}\setminus \{z_{1}=0\}$
with $\omega_{i} \rightarrow \omega_{\infty} \,\,(i \rightarrow \infty)$ in $S^{2n-1}$.
Let $\{k_{i}\}_{i}$ be a sequence of positive real numbers satisfying $k_{i} \rightarrow \infty\,\,(i\rightarrow \infty)$.
Set
\begin{equation}
\begin{aligned}
&\gamma_{i} :=(z_{i},\,t_{i}) \in E\setminus \{z_{1}=0\})\times T, \quad z_{i} :=k_{i}\omega_{i}\quad (i=1,\,\,2,\,\dots),\\
&\gamma_{\infty}:= (z_{\infty},\,t_{\infty})\in (S^{2n-1}\setminus \{z_{1}=0\})\infty \times T,\quad z_{\infty} := \omega_{\infty}\infty.
\end{aligned}
\end{equation}
Assume $\gamma_{i} \rightarrow \gamma_{\infty}\,\,(i \rightarrow \infty)$ in $\hat{X}$.
Then let us show  $\Phi(\gamma_{i}) \rightarrow \phi(\gamma_{\infty}) \,\,(i\rightarrow \infty)$ in $\hat{Y}$.
Note that we have
\begin{equation}
\begin{aligned}
&\Phi(\gamma_{i})=\left(k_{i}\omega_{i,\,1},\, \frac{\omega_{i,\,2}}{\omega_{i,\,1}},\,\dots,\, \frac{\omega_{i,\,n}}{\omega_{i,\,1}},\,t_{i} \right)\in \mathbb{C}\times \mathbb{C}^{n-1}\times T, \\
&\phi(\gamma_{\infty})= \left(\frac{\omega_{\infty,\,1}}{\vert \omega_{\infty,\,1} \vert}\infty,\,\frac{\omega_{\infty,\,2}}{\omega_{\infty,\,1}},\,\dots,\, \frac{\omega_{\infty,\,n}}{\omega_{\infty,\,1}},\,t_{\infty}   \right)\in S^{1}\infty \times \mathbb{C}^{n-1}\times T.
\end{aligned}
\end{equation}
As $\gamma_{i} \rightarrow \gamma_{\infty}\,\,(i \rightarrow \infty)$ in  $\hat{X}$, we have
\begin{equation}
\begin{aligned}
&\vert  k_{i}\omega_{i,\,j} - \vert z_{i} \vert \omega_{\infty,\,j} \vert \leq \epsilon_{i} \vert z_{i}  \vert, \quad \vert t_{i} -t_{\infty} \vert \leq \epsilon_{i}  \quad( j=1,\,\dots,\,n)\\
\end{aligned}
\end{equation}
for $\epsilon_{i}>0$ with $\epsilon_{i} \rightarrow 0$ $(i\rightarrow \infty)$.
Hence we obtain $\vert \omega_{i,\,j} - \omega_{\infty,\,j} \vert \leq \epsilon_{i}\,\,( j=1,\,\dots,\,n)$, from which we have, as $i\rightarrow \infty$, 
\begin{equation}\label{s1}
\begin{aligned}
\left\vert \frac{ k_{i}\omega_{i,\,1} }{\vert k_{i}\omega_{i,\,1} \vert} - \frac{\omega_{\infty,\,1}}{ \vert \omega_{\infty,\,1} \vert} \right\vert \rightarrow 0, \quad \left\vert \frac{\omega_{i,\,j}}{\omega_{i,\,1}} - \frac{\omega_{\infty,\,j}}{\omega_{\infty,\,1}} \right\vert \rightarrow 0 \quad(j=2,\,3,\,\dots n).
\end{aligned}
\end{equation}
Therefore we get $\Phi(\gamma_{i}) \rightarrow \phi(\gamma_{\infty})$ $(i\rightarrow \infty)$ in $\hat{Y}$.
Furthermore, let $\eta_{i} \,\,(i=1,\,2,\,\dots)$ and $\eta_{\infty}$ be points in $S^{1}$ satisfying
$\eta_{i} \rightarrow \eta_{\infty}\,\,(i\rightarrow \infty)$ in $S^{1}$.
Set
\begin{equation}
\begin{aligned}
&\rho_{i}:=(k_{i}\eta_{i},\,w_{i},\,t_{i})\in \mathbb{C}\times \mathbb{C}^{n-1}\times T \quad(i=1,\,2,\,\dots),\\
&\rho_{\infty}:=(\eta_{\infty}\infty,\,w_{\infty},\,t_{\infty})\in S^{1}\infty \times \mathbb{C}^{n-1}\times T.
\end{aligned}
\end{equation}
Suppose $\rho_{i} \rightarrow \rho_{\infty} \,\,(i \rightarrow \infty)$ in $\hat{Y}$.
Then let us show $\Phi^{-1}(\rho_{i}) \rightarrow \phi^{-1}(\rho_{\infty}) \,\,(i\rightarrow \infty)$ in $\hat{X}$.
Note that we have
\begin{equation}
\begin{aligned}
&\Phi^{-1}(\rho_{i}) = (k_{i}\eta_{i},\,k_{i}\eta_{i}w_{i},\,t_{i}) \in E\times T,\\
&\phi^{-1}(\rho_{\infty}) =\left(\left(\frac{\eta_{\infty}}{\sqrt{1+\vert w_{\infty}\vert^{2}}},\,\frac{\eta_{\infty}w_{\infty}}{\sqrt{1+\vert w_{\infty}\vert^{2}}} \right)\infty,\,t_{\infty}\right)\in S^{2n-1}\infty \times T.
\end{aligned}
\end{equation}
Since $\rho_{i} \rightarrow \rho_{\infty} \,\,(i \rightarrow \infty)$ in $\hat{Y}$, we have
\begin{equation}
\begin{aligned}
\vert \eta_{i} - \eta_{\infty} \vert \leq \epsilon_{i},\quad \vert w_{i}-w_{\infty}\vert \leq \epsilon_{i},\quad \vert t_{i}-t_{\infty}\vert \leq \epsilon_{i}
\end{aligned}
\end{equation}
for $\epsilon_{i}>0$ with $\epsilon_{i} \rightarrow 0$ $(i\rightarrow \infty)$, from which we have
\begin{equation}
\left\vert \frac{(k_{i}\eta_{i},\,k_{i}\eta_{i}w_{i})}{\vert (k_{i}\eta_{i},\,k_{i}\eta_{i}w_{i})\vert} - \frac{(\eta_{\infty},\,\eta_{\infty}w_{\infty})}{\sqrt{1+\vert w_{\infty}\vert^{2}}} \right\vert =\left\vert \frac{(\eta_{i},\,\eta_{i}w_{i})}{\sqrt{1+\vert w_{i}\vert^{2}}} -\frac{ (\eta_{\infty},\,\eta_{\infty} w_{\infty})}{\sqrt{1+\vert w_{\infty}\vert^{2}}} \right\vert \longrightarrow 0.
\end{equation}
Therefore we obtain $\Phi^{-1}(\rho_{i}) \rightarrow \phi^{-1}(\rho_{\infty}) \,\,(i\rightarrow \infty)$ in $\hat{X}$.
\end{proof}

By the lemma, we have
\begin{equation}\label{eqq-2}
\begin{aligned}
\mathscr{H}^{k}_{\overline{N}}(\hexpo)_{p_{\infty}} &= \underset{\epsilon}{\varinjlim} \mathrm{H}^{k}_{\overline{N} \cap \Omega_{\epsilon}}( \Omega_{\epsilon} ,\,\hexpo)\\&\cong \underset{\epsilon}{\varinjlim} 
\mathrm{H}^{k}_{ (\overline{\mathbb{R}}\times \mathbb{R}^{n-1} \times T) 
\cap O_{\epsilon}}( O_{\epsilon} ,\,\hexpoy)\\ &= 
\mathscr{H}^{k}_{\overline{\mathbb{R}}\times \mathbb{R}^{n-1} \times T }(\hexpoy)_{q_{\infty}}.
\end{aligned}
\end{equation}
Therefore, to prove the theorem, it suffices to show
\begin{equation}{\label{eq:first-taget}}
\mathscr{H}^{k}_{\overline{\mathbb{R}}\times \mathbb{R}^{n-1} \times T }(\hexpoy)_{q_{\infty}}
=0 \qquad(k \neq n).
\end{equation}
Let $\varphi_{\epsilon}$  be the polynomial of $w$ defined by
\begin{equation}
\varphi_{\epsilon}(w)={\left(\frac{ \epsilon}{2}\right)}^{2} - (w_{1}^{2} + \dots + w_{n-1}^{2}) \qquad (\epsilon>0).
\end{equation}
By using $\varphi_{\epsilon}$, we set
\begin{equation}
\begin{aligned}
&\widetilde{W}_{\epsilon} := \left\{ (w_{1},\,\dots,\,w_{n-1}) \in W_\epsilon; 
\,\,\mathrm{Re}(\varphi_{\epsilon}(w)) >0 \right\},\\
&\widetilde{O}_{\epsilon} :={\overline{V_{\epsilon}}}^\circ \times \widetilde{W}_{\epsilon} \times T_{\epsilon} \subset \hat{Y}.
\end{aligned}
\end{equation}
Obviously, the family $\{\widetilde{O}_{\epsilon}\}_{\epsilon>0}$ thus  defined becomes a fundamental system of neighborhoods of $q_{\infty}$ in $\hat{Y}$.
Hence, to obtain \eqref{eq:first-taget}, it is enough to prove
\begin{equation}
\mathrm{H}^{k}_{ (\overline{\mathbb{R}}\times \mathbb{R}^{n-1} \times T) 
\cap \widetilde{O}_{\epsilon}}( \widetilde{O}_{\epsilon} ,\,\hexpoy)=0
\qquad(k \neq n).
\end{equation}
Choose two points $a,\,b \in \mathbb{R}$ with ${\epsilon}^{-1}<a \leq b<+\infty$, and set
$S_{1} :=({\epsilon}^{-1},\,b ]$ and $S_{2} :=[a,\,+\infty]$.
By noticing
$$
(\overline{\mathbb{R}}\times \mathbb{R}^{n-1} \times T) 
\cap \widetilde{O}_{\epsilon} = (S_{1}\cup S_{2})\times 
(\widetilde{W}_{\epsilon} \cap\mathbb{R}^{n-1}) 
\times T_{\epsilon}, 
$$
we consider the long exact sequence of  cohomology groups
\begin{equation}\label{exse}
\begin{aligned}
\cdots&\rightarrow \mathrm{H}^{k}_{ (S_{1}\cap S_{2})\times (\widetilde{W}_{\epsilon} \cap\mathbb{R}^{n-1}) \times T_{\epsilon}}(\wtO_{\epsilon},\,\hexpoy )\\
&\rightarrow   \mathrm{H}^{k}_{S_{1}\times (\widetilde{W}_{\epsilon} \cap\mathbb{R}^{n-1}) \times T_{\epsilon} }(\wtO_{\epsilon},\,\hexpoy ) \oplus  \mathrm{H}^{k}_{S_{2}\times (\widetilde{W}_{\epsilon} \cap\mathbb{R}^{n-1}) \times T_{\epsilon} }(\wtO_{\epsilon},\,\hexpoy )  \\
&\rightarrow \mathrm{H}^{k}_{ (\overline{\mathbb{R}}\times \mathbb{R}^{n-1} \times T) \cap \wtO_{\epsilon} }(\wtO_{\epsilon},\,\hexpoy )\rightarrow \cdots.
\end{aligned}
\end{equation}
Set $N' := \mathbb{R} \times \mathbb{R}^{n-1} \times T \subset Y$.
The following lemma is well-known for specialists in this direction:
Let $\mathcal{B}\mathcal{O}_{N'}$ denote
the sheaf $\mathscr{H}^{n}_{N'}(\mathcal{O}_{Y})$ 
of hyperfunctions with  holomorphic parameters on $N'$ (here we omit
the orientation sheaf for simplicity).
\begin{lemma}
We have
\begin{equation}{\label{eq:vanish-bo-1}}
\mathrm{H}^{k}_{ (S_{1}\cap S_{2})\times (\widetilde{W}_{\epsilon} \cap\mathbb{R}^{n-1}) \times T_{\epsilon}}(\wtO_{\epsilon},\,\hexpoy ) =0 \quad (k\neq n),
\end{equation}
and
\begin{equation}
\begin{aligned}
\mathrm{H}^{k}_{ S_{1}\times (\widetilde{W}_{\epsilon} \cap\mathbb{R}^{n-1}) \times T_{\epsilon}}(\wtO_{\epsilon},\,\hexpoy )
=0 \quad ( k \neq n).
\end{aligned}
\end{equation}
\end{lemma}
\begin{proof}
We only show \eqref{eq:vanish-bo-1}.
Since  $N' = \mathbb{R}^{n} \times T$ 
is purely $n$-codimensional relative to the sheaf $\mathcal{O}_{Y} $,
we have
\begin{equation}
\mathrm{H}^{k}_{ (S_{1}\cap S_{2})\times (\widetilde{W}_{\epsilon} \cap\mathbb{R}^{n-1}) \times T_{\epsilon}}(\wtO_{\epsilon},\,\hexpoy )
=\mathrm{H}^{k-n}_{ (S_{1}\cap S_{2})\times (\widetilde{W}_{\epsilon} \cap\mathbb{R}^{n-1})\times T_{\epsilon}}(\wtO_{\epsilon}\cap N',\,  \mathcal{B}\mathcal{O}_{N'}).
\end{equation}
Hence \eqref{eq:vanish-bo-1} holds for $k < n$. By considering a relative Leray covering,
we can show \eqref{eq:vanish-bo-1} for $k > n$ 
in the same way as that for the proof of Lemma \ref{lemma3.1}.
\end{proof}

Then it follows from the above lemma and $(\ref{exse})$ that we get
\begin{equation}\label{nen}
\mathrm{H}^{k}_{ (\overline{\mathbb{R}}\times \mathbb{R}^{n-1} \times T) \cap \wtO_{\epsilon} }(\wtO_{\epsilon},\,\hexpoy)=\mathrm{H}^{k}_{ S_{2}\times (\widetilde{W}_{\epsilon} \cap\mathbb{R}^{n-1}) \times T_{\epsilon}}(\wtO_{\epsilon},\,\hexpoy )\quad (k \neq n-1,\,n).
\end{equation}
We also get the following exact sequence by $(\ref{exse})$.
\begin{equation}\label{65}
\begin{aligned}
0 &\rightarrow \mathrm{H}^{n-1}_{ S_{2}\times (\widetilde{W}_{\epsilon} \cap\mathbb{R}^{n-1}) \times T_{\epsilon}}(\wtO_{\epsilon},\,\hexpoy )\\
&\rightarrow \mathrm{H}^{n-1}_{ (S_{1}\cup S_{2})\times (\widetilde{W}_{\epsilon} \cap\mathbb{R}^{n-1}) \times T_{\epsilon}}(\wtO_{\epsilon},\,\hexpoy )\\
&\rightarrow  \Gamma_{ (S_{1}\cap S_{2}) \times 
(\widetilde{W}_{\epsilon}\cap\mathbb{R}^{n-1}) \times T_{\epsilon} }(\wtO_{\epsilon}\cap N',\,\mathcal{B}\mathcal{O}_{N'})\\
&\overset{\iota} {\rightarrow}    
\Gamma_{S_{1}\times (\widetilde{W}_{\epsilon} \cap\mathbb{R}^{n-1})\times T_{\epsilon} }(\wtO_{\epsilon}\cap N',\,\mathcal{B}\mathcal{O}_{N'})
\oplus \mathrm{H}^{n}_{S_{2}\times (\widetilde{W}_{\epsilon} \cap\mathbb{R}^{n-1}) 
\times T_{\epsilon} }(\wtO_{\epsilon},\,\hexpoy ).
\end{aligned}
\end{equation}
Since the canonical morphism 
$$ \Gamma_{ (S_{1}\cap S_{2}) \times (\widetilde{W}_{\epsilon}\cap\mathbb{R}^{n-1}) \times T_{\epsilon} }(\wtO_{\epsilon}\cap N',\,\mathcal{B}\mathcal{O}_{N'})
\rightarrow    \Gamma_{S_{1}\times (\widetilde{W}_{\epsilon} \cap\mathbb{R}^{n-1})\times T_{\epsilon} }(\wtO_{\epsilon}\cap N',\,\mathcal{B}\mathcal{O}_{N'})$$
is injective,
the morphism $\iota$ in $(\ref{65})$ is injective.
As a result, $(\ref{nen})$ holds for $k=n-1$ also.
Hence, together with $(\ref{nen})$, it suffices to show
\begin{equation}\label{van}
\mathrm{H}^{k}_{S_{2} \times (\widetilde{W}_{\epsilon} \cap\mathbb{R}^{n-1})\times T_{\epsilon}}(\wtO_{\epsilon},\,\hexpoy) =0 \qquad( k\neq n).
\end{equation}
By Lemma $\ref{lemma3.1}$, we have $(\ref{van})$ for $k\geq n+1$.
Let us prove $(\ref{van})$ to be true for $0\leq k \leq n-1$.
Set
\begin{equation}
\begin{aligned}
&L_{\epsilon} := \left\{ w \in \mathbb{C}^{n-1}\, ; \,\,|w_{1}| \leq \epsilon,\,\dots,|w_{n-1}| \leq \epsilon \right\} \cap \mathbb{R}^{n-1},\\
&K_{\epsilon} := L_{\epsilon} \cap \left\{ w\in \mathbb{C}^{n-1}\,\, ; \,\,\mathrm{Re}(\varphi_{\epsilon}(w)) \leq  0  \right\}.
\end{aligned}
\end{equation}
Then we have
\begin{equation}
L_{\epsilon} \setminus K_{\epsilon} = \widetilde{W}_{\epsilon}  \cap \mathbb{R}^{n-1}.
\end{equation}
Since $L_{\epsilon}$ and $K_{\epsilon} $ are closed analytic polyhedra,
it follows from Theorem \ref{polyhedra} that we obtain
\begin{equation}
\begin{aligned}
\mathrm{H}^{k}_{S_{2} \times (\widetilde{W}_{\epsilon} \cap\mathbb{R}^{n-1})\times T_{\epsilon}}(\wtO_{\epsilon},\,\hexpoy) 
&= \mathrm{H}^{k}_{S_{2} \times (L_{\epsilon} \setminus K_{\epsilon})\times T_{\epsilon} }(\rcsptwo\times \mathbb{C}^{n-1}\times T_{\epsilon},\,\hexpoy)\\
&=0 \qquad (0\leq k \leq n-1).
\end{aligned}
\end{equation}
This completes the proof.
\end{proof}

The following theorem plays an important role in proving softness of the sheaf of Laplace hyperfunctions with respect to the variables of hyperfunction part.
\begin{theorem}\label{edge2}
The closed subset $\partial N$ is purely n-codimensional relative to the sheaf $\hexpo$, that is, we have
\begin{equation}
\mathscr{H}^{k}_{\partial N} (\hexpo) = 0 \qquad (k \ne n).
\end{equation}
\end{theorem}
\begin{proof}
It is sufficient to compute the stalk of $\mathscr{H}^{k}_{\partial N}(\hexpo)$ at 
$p_{\infty} = (1,\, 0,\, \dots,\, 0)\infty\times(0)$ $\in \partial N \subset \hat{X}$.
We use the same notations as those in the proof of Theorem $\ref{thm-codimentionality}$.
By the same reasoning as that in the proof of Theorem $\ref{thm-codimentionality}$, we have
\begin{equation}\label{eqq}
\begin{aligned}
\mathscr{H}^{k}_{\partial N}(\hexpo)_{p_{\infty}}  &\cong \mathscr{H}^{k}_{
\{+\infty\}\times \mathbb{R}^{n-1}\times {T} }( \hexpoy)_{q_{\infty}} \\
& =\underset{\epsilon  \downarrow 0 }{\varinjlim}\,\, \mathrm{H}^{k}_{ \{ +\infty \} \times (W_{\epsilon} \cap  \mathbb{R}^{n-1}) \times{T_{\epsilon}}} (O_{\epsilon} ,\,  \hexpoy).
\end{aligned}
\end{equation}
Let us show
\begin{equation}\label{k-coho}
\mathrm{H}^{k}_{\{+\infty\}\times (W_{\epsilon} \cap  \mathbb{R}^{n-1}) \times{T_{\epsilon}}} (O_{\epsilon} ,\,  \hexpoy) =0 \quad (k\neq n).
\end{equation}
We have, by Theorem \ref{thm-codimentionality},
\begin{equation}
\mathscr{H}^{k}_{\overline{\mathbb{R}} \times \mathbb{R}^{n-1} \times T}
(\hexpoy) = 0 \quad (k\neq n),
\end{equation}
from which we get
\begin{equation}
\mathrm{H}^{k}_{\{+\infty\}\times (W_{\epsilon} \cap  \mathbb{R}^{n-1}) \times{T_{\epsilon}}} (O_{\epsilon} ,\,  \hexpoy)
=
\mathrm{H}^{k-n}_{\{+\infty\}\times (W_{\epsilon} \cap  \mathbb{R}^{n-1}) \times{T_{\epsilon}}} 
(O_{\epsilon} ,\,  
\mathscr{H}^{n}_{\overline{\mathbb{R}} \times \mathbb{R}^{n-1} \times T}
(\hexpoy)).
\end{equation}
Hence $(\ref{k-coho})$ follows for $k \le n-1$. We also obtain
$(\ref{k-coho})$ for $k\geq n+1$ by Lemma $\ref{lemma3.1}$.
This completes the proof.
\end{proof}

As a particular case, we have the following corollaries.
\begin{corollary}\label{edge-Rn}
The closed subset $\overline{M}$ in  $\rcsp$  
is purely n-codimensional relative to 
the sheaf $\mathcal{O}^{\exp}_{\rcsp}$, i.e.,
\begin{equation}
\mathscr{H}^{k}_{\overline{M}}(\mathcal{O}^{\exp}_{\rcsp})=0 \qquad (k \ne n).
\end{equation}
\end{corollary}
\begin{corollary}\label{edge}
The boundary $\partial M := \overline{M} \setminus M$ 
of $M$ is purely n-codimensional relative to the sheaf 
$\mathcal{O}^{\exp}_{\rcsp}$, i.e.,
\begin{equation}
\mathscr{H}^{k}_{\partial M} (\mathcal{O}^{\exp}_{\rcsp}) = 0 \qquad (k \ne n).
\end{equation}
\end{corollary}

At the end of this section, 
we give a lemma for vanishing of higher cohomology groups of global sections of
$\pboexp$,
which can be proved by using the same argument 
as that in the proof of Theorem \ref{edge2}.
\begin{lemma}{\label{lem:vanishing_on_hemisphere}}
Let $V$ be a Stein open subset in $T$, and let $\Omega$ be an open
subset in $\partial M \subset \rcsp$ which is properly contained in
an open hemisphere of $S^{n-1} = \partial M$.
Then we have
\begin{equation}
\mathrm{H}^{k}(\Omega \times V,\, \pboexp) = 0 
\qquad (k \ne 0).
\end{equation}
\end{lemma}
\begin{proof}
We may assume that $\Omega$ is properly contained in 
$\{\operatorname{Re} z_1 > 0\}$ of $S^{n-1}$.
Then there exists a relatively compact open subset
$\widetilde{\Omega} \subset \mathbb{R}^{n-1}$ for which 
$\Omega \times V$ is homeomorphic to 
$\{+\infty\} \times \widetilde{\Omega} \times V \subset \hat{Y}$ 
by the map $\widehat{\Phi} = \Phi \sqcup \phi$. 
It follows from Grauert's theorem that there exists a Stein open neighborhood
$W \subset \mathbb{C}^{n-1}$ of $\widetilde{\Omega}$ satisfying
$\widetilde{\Omega} = W \cap \mathbb{R}^{n-1}$.
Then we have
$$
\mathrm{H}^{k}(\Omega \times V,\, \pboexp)
= \mathrm{H}^{k+n}_{ \{ +\infty \} \times \widetilde{\Omega} \times V} 
(\rcsptwo \times W \times V ,\,  \hexpoy).
$$
Hence the claim follows from Lemma \ref{lemma3.1}.
\end{proof}

\section{The sheaf of Laplace hyperfunctions in several variables}
In this section we define the sheaf $\mathcal{B}\mathcal{O}^{\exp}_{\overline{N}}$ of Laplace hyperfunctions with holomorphic parameters.
We see that every hyperfunction can be extend to a Laplace hyperfunction, and
we also  show that there exists the canonical embedding from real analytic functions of exponential type to Laplace hyperfunctions as  a sheaf morphism.

Let us recall the geometrical situation studied in the previous section:
Let $M$ be an $n$-dimensional $\mathbb{R}$-vector space ($n \ge 1$) 
with an inner product, 
and let $E$ be its complexification $\mathbb{C} \underset{\mathbb{R}}{\otimes} M$
which is an $n$-dimensional $\mathbb{C}$-vector space with the norm induced from
the inner product of $M$.
We denote by $\overline{M}$ the closure of $M$ in the radial compactification
$\rcsp$ of $E$ and set $\partial M := \overline{M} \setminus M$.
We also set $T:=\mathbb{C}^m$ ($m \ge 0$) and 
$$
\begin{array}{ccc}
N := M \times T &\subset &X := E \times T \\
\cap & & \cap \\
\overline{N} := \overline{M} \times T & \subset & \hat{X}:= \rcsp \times T.
\end{array}
$$
Define $\partial N := \overline{N} \setminus N = \partial M \times T$.
Let $\mathbb{Z}_{\overline{N}}$ 
be the constant sheaf on $\overline{N}$ with stalk $\mathbb{Z}$.

\begin{proposition}\label{constant}
\begin{equation}
\mathscr{H}^{k}_{\overline{N}}(\mathbb{Z}_{{\hat{X}}}) \cong
\left\{
\begin{array}{ll}
\mathbb{Z}_{\overline{N}} &\quad (k=n), \\
0 &\quad (k\neq n).\\
\end{array}
\right.
\end{equation}
\end{proposition}

\begin{proof}
Let us compute the stalk of 
$\mathscr{H}^{k}_{\overline{N}}(\mathbb{Z}_{\hat{X}})$ at $p \in \overline{N}$.
It suffices to check the stalk at $p_{\infty}=(1,\,0,\,\dots,\,0)\infty \times (0)\in 
\partial N \subset \overline{N}$. 
Let $B^{n}$ be the closed unit ball in $\mathbb{R}^{n} \,\,(n\geq 1)$.
Then we have $\hat{X} \cong B^{2n} \times T$ and 
$\overline{N} \cong B^{n} \times T$ topologically. 
We denote by $\mathbb{R}_{\leq 0}$ the set of non-positive real numbers.
There is a family $\{U_{\epsilon} \}_{\epsilon}$ 
$($resp. $\{V_{\epsilon} \times W_{\epsilon}\}_{\epsilon})$ 
of  fundamental open neighborhoods of the point $(1,\,0,\,\dots,\,0) \in B^{2n}$  
$($resp. $(0)\times (0)\in \mathbb{R}_{\leq 0}\times \mathbb{R}^{2n-1})$ 
satisfying $B^{2n}\cap U_{\epsilon} \cong  (\mathbb{R}_{\leq 0}\times \mathbb{R}^{2n-1}) \cap (V_{\epsilon} \times W_{\epsilon})$ 
and $B^{n}\cap U_{\epsilon} \cong  (\mathbb{R}_{\leq 0}\times \mathbb{R}^{n-1}) \cap (V_{\epsilon} \times W_{\epsilon})$ topologically.
Hence, for a family $\{T_{\epsilon}\}_{\epsilon}$  of  fundamental neighborhoods of the origin in $T$, we obtain
\begin{equation}
\begin{aligned}
\mathscr{H}^{k}_{\overline{N}}(\mathbb{Z}_{\hat{X}})_{p_{\infty}}
&\cong \mathscr{H}^{k}_{B^{n}\times  T}(\mathbb{Z}_{B^{2n} \times T})_{(1,\,0,\,\dots,\,0)\times (0)}\\
&\cong \mathscr{H}^{k}_{\mathbb{R}_{\leq 0} \times \mathbb{R}^{n-1} \times T}(\mathbb{Z}_{\mathbb{R}_{\leq 0} \times \mathbb{R}^{2n-1}\times T})_{(0)\times (0)\times (0)} \\
&=\underset{\epsilon  }{\varinjlim}\,\, \mathrm{H}^{k}_{(\mathbb{R}_{\leq 0} \times \mathbb{R}^{n-1} \times T) \cap (V_{\epsilon}\times W_{\epsilon}\times T_{\epsilon})} (V_{\epsilon}\times W_{\epsilon}\times T_{\epsilon} ,\,  \mathbb{Z}_{\mathbb{R}_{\leq 0} \times \mathbb{R}^{2n-1}\times T})
\end{aligned}
\end{equation}
for $k \in \mathbb{Z}$.
As we may assume that $V_{\epsilon}$ and $T_{\epsilon}$ are contractible, we obtain
\begin{equation}
\begin{aligned}
\mathscr{H}^{k}_{\mathbb{R}_{\leq 0} \times \mathbb{R}^{n-1} \times T}(\mathbb{Z}_{\mathbb{R}_{\leq 0} \times \mathbb{R}^{2n-1}\times T})_{(0)\times (0)\times (0)}
&\cong \mathscr{H}^{k}_{\mathbb{R}^{n-1}}(\mathbb{Z}_{\mathbb{R}^{2n-1}})_{(0)}.\\
\end{aligned}
\end{equation}
Let $W_{\epsilon}$ be the open ball in $\mathbb{R}^{2n-1}$ of center the origin
and radius $\epsilon > 0$.
We consider the following long exact sequence of  cohomology groups.
\begin{equation}
\rightarrow \mathrm{H}^{k}_{\mathbb{R}^{n-1}\cap W_{\epsilon}}(W_{\epsilon},\mathbb{Z}_{\mathbb{R}^{2n-1}}) 
\rightarrow \mathrm{H}^{k}(W_{\epsilon},\,\mathbb{Z}_{\mathbb{R}^{2n-1}})  
\rightarrow \mathrm{H}^{k}({W}_{\epsilon}\setminus \mathbb{R}^{n-1},\,\mathbb{Z}_{\mathbb{R}^{2n-1}}) \rightarrow.
\end{equation}
By noticing
\begin{equation}
\mathrm{H}^{k}(W_{\epsilon},\,\mathbb{Z}_{\mathbb{R}^{2n-1}})   =
\left\{
\begin{array}{ll}
\mathbb{Z} &\quad k=0, \\
0 &\quad k\neq 0,\\
\end{array}
\right.
\end{equation}
and
\begin{equation}
\mathrm{H}^{k}({W}_{\epsilon}\setminus \mathbb{R}^{n-1},\,\mathbb{Z}_{\mathbb{R}^{2n-1}}) = \mathrm{H}^{k}(S^{n-1},\,\mathbb{Z}_{S^{n-1}})=
\left\{
\begin{array}{ll}
\mathbb{Z} &\quad k=0\,\,\mathrm{or}\,\, n-1, \\
0 &\quad \mathrm{otherwise},\\
\end{array}
\right.
\end{equation}
we have 
\begin{equation}
\mathrm{H}^{k}_{\mathbb{R}^{n-1}\cap W_{\epsilon}}(W_{\epsilon},\mathbb{Z}_{\mathbb{R}^{2n-1}})  =
\left\{
\begin{array}{ll}
\mathbb{Z} &\quad k=n, \\
0 &\quad k\neq n.\\
\end{array}
\right.
\end{equation}
Therefore we have obtained the assertion.
\end{proof}

For simplicity of notation, we write $\omega_{\overline{N}}$ instead of $\mathscr{H}^{n}_{\overline{N}}(\mathbb{Z}_{{\hat{X}}})$.
If $\Omega$ is a connected open subset in $\overline{N}$,
we can regard $\omega_{\overline{N}}(\Omega)$ as $\Z$ by the above proposition.

\begin{definition}
We define the sheaf $\mathcal{B}\mathcal{O}^{exp}_{\overline{N}}$ of Laplace hyperfunctions with holomorphic parameters on $\overline{N}$ by
\begin{equation}
\mathcal{B}\mathcal{O}^{\exp}_{\overline{N}} := \mathscr{H}^{n}_{\overline{N}}(\mathcal{O}^{\exp}_{\hat{X}}) \underset{\mathbb{Z}_{\overline{N}}}{\otimes} \omega_{\overline{N}}.
\end{equation}
\end{definition}

As $ \mathscr{H}^{k}_{\overline{N}}(\hexpo) =0 \,\,(k<n)$ holds by Theorem \ref{thm-codimentionality}, a global section of the sheaf $\mathcal{B}\mathcal{O}^{\exp}_{\overline{N}}$  can be written in terms of  cohomology groups.
For a connected open subset $\Omega$ in $\overline{N}$, 
by taking an open subset $V$ in $\hat{X}$ with
$\Omega = \overline{N} \cap V$, we have
\begin{equation}
\mathcal{B}\mathcal{O}^{\exp}_{\overline{N}} (\Omega ) 
= \mathrm{H}^{n}_{\Omega}(V,\,\hexpo) \underset{\mathbb{Z}}{\otimes}
\omega_{\overline{N}}(\Omega).
\end{equation}
Note that the above representation does not depend on the choice of $V$ in $\hat{X}$.
The restriction to $N$ of the sheaf $\mathcal{B} \mathcal{O}^{\exp}_{\overline{N}}$ is evidently just the sheaf $\mathcal{B}\mathcal{O}_{N}$ of hyperfunctions with holomorphic parameters.

As a particular case, we have the following definition.
\begin{definition}
The sheaf of Laplace hyperfunctions on $\overline{M}$ is defined by
\begin{equation}
\mathcal{B}^{\exp}_{\overline{M}} := 
\mathscr{H}^{n}_{\overline{M}}(\mathcal{O}^{\exp}_{\rcsp}) 
\underset{\mathbb{Z}_{\overline{M}}}{\otimes} \omega_{\overline{M}}.
\end{equation}
\end{definition}

The following theorem states that every hyperfunction can be extended to a Laplace hyperfunction.
Let $j : N \hookrightarrow \overline{N}$ be the natural embedding.
\begin{theorem}
The canonical sheaf morphism 
$\mathcal{B}\mathcal{O}^{\exp}_{\overline{N}} 
\rightarrow j_{*}\mathcal{B}\mathcal{O}_{N}$ is surjective.
\end{theorem}
\begin{proof}
Let us consider the following distinguished triangle.
\begin{equation}
\mathbf{R}\Gamma_{\partial N}(\mathbf{R}\Gamma_{\overline{N}}(\mathcal{O}^{\exp}_{\hat{X}}))
\rightarrow
\mathbf{R}\Gamma_{\overline{N}}(\mathcal{O}^{\exp}_{\hat{X}})
\rightarrow
\mathbf{R}j_{*}j^{-1}\mathbf{R}\Gamma_{\overline{N}}(\mathcal{O}^{\exp}_{\hat{X}})
\overset{+1}{\rightarrow}.
\end{equation}
Then the assertion follows from Theorem \ref{edge2} and $\mathbf{R}j_{*}j^{-1}\mathbf{R}\Gamma_{\overline{N}}(\mathcal{O}^{\exp}_{\hat{X}}) = j_{*}\mathcal{B}\mathcal{O}_{N}[-n]$.
\end{proof}

Let $i: \overline{N} \hookrightarrow \hat{X}$ be the natural embedding.
\begin{definition}
The sheaf of real analytic functions of exponential type on $\overline{N}$ with holomorphic parameters is defined by
\begin{equation}
\mathcal{A}\mathcal{O}^{\exp}_{\overline{N}} := i^{-1}\hexpo =\hexpo|_{\overline{N}}.
\end{equation}
\end{definition}
Let us see that real analytic functions of exponential type with holomorphic parameters are regarded as Laplace hyperfunctions with holomorphic parameters.
We consider the following morphism.
\begin{equation}
	i^{-1}\hexpo \underset{\mathbb{Z}_{\overline{N}}}{\otimes} \omega_{\overline{N}} [-n] \rightarrow i^{!}\hexpo
\simeq i^{-1} \mathbf{R}\Gamma_{\overline{N}}(\hexpo).
\end{equation}
Applying the shift functor $[n]$ and the functor $(\cdot  )\otimes \omega_{\overline{N}} $ to the above morphism,  we obtain the sheaf morphism   $\alpha  :\mathcal{A}\mathcal{O}^{\exp}_{\overline{N}} \rightarrow \mathcal{B}\mathcal{O}^{\exp}_{\overline{N}}$.
Let $\mathcal{A}\mathcal{O}_{N}$ be the sheaf of  real analytic functions with holomorphic parameters on $N$,
 and let $\beta :\mathcal{A}\mathcal{O}_{N} \rightarrow \mathcal{B}\mathcal{O}_{N}$ be the canonical embedding of sheaves.
Now we consider the commutative diagram of sheaf morphisms.
\begin{equation}
\begin{CD}
\mathcal{A}\mathcal{O}^{\exp}_{\overline{N}} @> \alpha >> \mathcal{B}\mathcal{O}^{\exp}_{\overline{N}}\\
@V VV @VV V \\
j_{*}\mathcal{A}\mathcal{O}_{N} @> j_{*} \beta >> j_{*}\mathcal{B}\mathcal{O}_{N}
\end{CD}.
\end{equation}
Since  the morphisms $\mathcal{A}\mathcal{O}^{\exp}_{\overline{N}} \rightarrow j_{*}\mathcal{A}\mathcal{O}_{N}$ and $j_{*} \beta$ 
in the above diagram are injective, we have the following theorem.
\begin{theorem}
The canonical sheaf morphism $\mathcal{A}\mathcal{O}^{\exp}_{\overline{N}} \rightarrow \mathcal{B}\mathcal{O}^{\exp}_{\overline{N}}$ is injective.
\end{theorem}

\section{Several properties of the sheaf $\boexp$  }

In this section,
we prove softness of the sheaf $\boexp$ with respect to the variables of hyperfunction
part, and we also  show surjectivity of the restriction morphism 
$$
\boexp(\Omega \times V) \rightarrow 
j_*\mathcal{B}\mathcal{O}_{N}( \Omega \times V ) =
\mathcal{B}\mathcal{O}_{N}( (\Omega \times V ) \cap N)
$$ 
for any open subset $\Omega$ in $\overline{M}$
and a Stein open subset $V$ in $T$.
Here $j: N \hookrightarrow \overline{N}$ denotes the canonical embedding.

\begin{proposition}{\label{prop:flabby-dim-partial_R}}
Let $V$ be a Stein open subset in $T$, 
and let $\pi_V: \rcsp \times V \to \rcsp$ denote 
the canonical projection. Then we have 
$$
\mathbf{R}\pi_{V*}\pboexp \simeq \pi_{V*}\pboexp.
$$
Further, the flabby dimension of $\pi_{V*}\pboexp$ is less than or equal to $1$.
\end{proposition}
\begin{proof}
Let $\{\Omega_i\}_i$ be a finite family of open subsets $\Omega_i$ in 
$\partial M$ satisfying that 
$\cup_i \Omega_i = \partial M$ and each $\Omega_i$ is properly contained
in some hemisphere of $S^{n-1} = \partial M$.
Then it follows from Lemma \ref{lem:vanishing_on_hemisphere} that, for
any open subset $\Omega \subset \Omega_i$, we have
\begin{equation}
\mathrm{H}^{k}(\Omega,\, \mathbf{R}\pi_{V*}\pboexp) =
\mathrm{H}^{k}(\Omega \times V,\,
\pboexp) = 0 
\qquad (k \ne 0).
\end{equation}
Hence the complex $\mathbf{R}\pi_{V*}\pboexp$ is concentrated in degree $0$, and
we have
$$
\mathrm{H}^{k}(\Omega,\, \pi_{V*}\pboexp) = 0 \qquad (k \ne 0)
$$
for any open subset $\Omega \subset \Omega_i$.
This concludes that, for each $i$,
the flabby dimension of the sheaf $\pi_{V*}\pboexp\big|_{\Omega_i}$ on $\Omega_i$
is less than or equal to $1$. Then the last claim is a consequence
of the following easy lemma.
\end{proof}
\begin{lemma}
Let $\mathcal{F}$ be a sheaf on a topological space $X$,
and let $\{\Omega_i\}$ be an open covering of $X$. Assume that, for each $i$, 
	the flabby dimension of the sheaf 
	$\mathcal{F}\big|_{\Omega_i}$ on $\Omega_i$ is less than or equal to $\ell \in \mathbb{N} \cup \{0\}$.
	Then, that of $\mathcal{F}$ on $X$ is also less than or equal to $\ell$.
\end{lemma}
\begin{proof}
It is easy to see that, if $\mathcal{F}\big|_{\Omega_i}$ is a flabby sheaf on $\Omega_i$
for each $i$, then $\mathcal{F}$ itself is flabby on $X$.
Now let us take a flabby resolution $L$ of $\mathcal{F}$ on $X$:
$$
L:\,0 \to \mathcal{F} \to \mathcal{L}^0 \overset{d_0}{\to}
\mathcal{L}^1 \overset{d_1}{\to} 
\mathcal{L}^2 \overset{d_2}{\to} \cdots.
$$
By the assumption, for each $i$, 
the restriction $\operatorname{Ker} d_\ell\big|_{\Omega_i}$ 
of the kernel sheaf $\operatorname{Ker} d_\ell$
is flabby. Hence
$\operatorname{Ker} d_\ell$ is flabby on $X$, which entails that 
$\mathcal{F}$ has a flabby resolution of length $\ell$.
\end{proof}

We prepare some notations which are needed in the subsequent arguments.
Let $(z_1, \dots, z_n)$ be a system of coordinates of $E$.
Let 
$\Phi : E \setminus \{z_{1} =0\}  
\rightarrow (\mathbb{C}\setminus\{0\}) \times \mathbb{C}^{n-1}$ 
be the holomorphic map given by 
\begin{equation}
\Phi(z_{1},\,\dots,\,z_{n}) =\left(z_1,\,\frac{z_{2}}{z_{1}},\,\dots,\,\frac{z_{n}}{z_{1}} \right),
\end{equation}
and let $A$ be a $\mathbb{C}$-linear isomorphism on $E$.
Then we define $\Phi_A := \Phi \circ A$ and
the linear hypersurface $H_A := A^{-1}(\{z_1 = 0\})$ in  $E$.
As in Subsection \ref{subsec:edge}, the $\Phi_A$ extends to a homeomorphism between 
$\rcsp \setminus \overline{H_A}$ 
and $(\rcsptwo \setminus \{0\}) \times \mathbb{C}^{n-1}$ 
which is denoted by the same symbol $\Phi_{A}$ hereafter.
\begin{definition}\label{phi}
Let $K$ be a compact subset in $\overline{M}$.
We say that $K$ is of product type in $\overline{M}$ 
if there exists a linear isomorphism $A$ on $E$ such that
$K$ and $S \times L_{1} \times \dots \times L_{n-1}$ are homeomorphic by $\Phi_A$
for a compact subset $S\subset \overline{\mathbb{R}}\setminus\{0\}$
and compact subsets $L_{1},\,\cdots,\,L_{n-1}$ in $\mathbb{R}$.
\end{definition}
Then we establish the following vanishing theorem for a closed subset
which is a union of compact subsets of product type in $\overline{M}$.
\begin{proposition}\label{cover}
Let $V$ be a Stein open subset in $T$, and
let $K$ be a finite union of closed subsets $K_{i}$ in $\partial {M}$ 
$(i=1,2,\dots,\ell)$.
Assume that each $K_{i} \subset \partial M$ is of product type in $\overline{M}$.
Then we have
\begin{equation}
\mathrm{H}^{k}_{K \times V}(\rcsp \times V,\,\hexpo)=0 \qquad (k\neq n).
\end{equation}
\end{proposition}
\begin{proof}
Let $\pi_V: \rcsp \times V \to \rcsp$ be the canonical 
projection.  On account of Theorem $\ref{edge2}$
and Proposition \ref{prop:flabby-dim-partial_R},
we have
\begin{equation}
\mathrm{H}^{k}_{K \times V}(\rcsp \times V,\,\hexpo)
=
\mathrm{H}^{k-n}_{K}(\rcsp,\, \pi_{V*}\pboexp) = 0
\end{equation}
if $k < n$ or $k > n+1$.
Let us show $\mathrm{H}^{n+1}_{K \times V} 
(\rcsp\times V,\,\hexpo) = 0$.
Since each $K_i$ is of product type in $\overline{M}$,
it follows from Lemma \ref{lemma3.1} that we have
\begin{equation}\label{puz}
\mathrm{H}^{k}_{K_{i} \times V} (\rcsp \times V,\,\hexpo)= 0
\quad (k > n)
\end{equation}
for each $i$.
Let us  first show $\mathrm{H}^{n+1}_{(K_{1} \cup K_{2}) \times V} 
(\rcsp \times V,\,\hexpo)=0$
from 
$\mathrm{H}^{n+1}_{K_{1} \times V} (\rcsp \times V,\,\hexpo) = 0$ 
and $\mathrm{H}^{n+1}_{K_{2}\times V} (\rcsp \times V,\,\hexpo) = 0$.
Consider the following long exact sequence of  cohomology groups
\begin{equation}
\begin{aligned}
&\rightarrow \mathrm{H}^{n+1}_{K_{1} \times V} (\rcsp \times V,\,\hexpo) 
\oplus \mathrm{H}^{n+1}_{K_{2} \times V} (\rcsp \times V,\,\hexpo) \\
&\rightarrow \mathrm{H}^{n+1}_{(K_{1} \cup K_{2}) \times V} 
(\rcsp \times V,\,\oexp)
\rightarrow \mathrm{H}^{n+2}_{(K_{1} \cap K_{2})\times V} 
(\rcsp \times V,\,\hexpo)\rightarrow.
\end{aligned}
\end{equation}
Then, by Proposition \ref{prop:flabby-dim-partial_R},
we have
$$
\mathrm{H}^{n+2}_{(K_{1} \cap K_{2}) \times V}\, 
(\rcsp \times V,\,\hexpo)
=
\mathrm{H}^{2}_{K_{1} \cap K_{2}}\, 
(\rcsp,\, \pi_{V*}\pboexp) = 0.
$$
Hence $\mathrm{H}^{n+1}_{(K_{1} \cup K_{2})\times V} 
(\rcsp\times V,\,\hexpo) = 0$
follows from the above long exact sequence.
By repeatedly applying the same argument to the sets $K_{3},\,K_{4},\,\dots,\,K_{\ell}$, 
we finally obtain
\begin{equation}
\mathrm{H}^{n+1}_{K \times V} (\rcsp \times V,\,\hexpo)=
\mathrm{H}^{n+1}_{(\bigcup_{i} K_i) \times V } 
(\rcsp \times V,\,\hexpo)=0.
\end{equation}
This completes the proof.
\end{proof}
As a corollary, we can obtain a global version of Theorem \ref{edge2}.
\begin{corollary}\label{boundary}
Let $V$ be a Stein open subset in $T$. Then
we have
\begin{equation}\label{partial}
\mathrm{H}^{k}_{\partial N} (\rcsp\times V,\,\hexpo) = 0\qquad (k \ne n).
\end{equation}
\end{corollary}

We also obtain the following proposition as a consequence of
Proposition \ref{cover}. 
\begin{proposition}\label{surje}
Let $V$ be a Stein open subset in $T$, and let $\pi_V: \rcsp \times V
\to \rcsp$ denote the canonical projection.
Then, for any  closed subset $K \subset \partial M$,
the restriction morphism
$$
\Gamma\left(\partial M,\,\pi_{V*}\pboexp\right)
\rightarrow 
\Gamma\left(K,\, \pi_{V*}\pboexp\right)
$$ 
is surjective. That is, the sheaf $\pi_{V*}\pboexp$ is soft.
\end{proposition}
\begin{proof}
Since $\partial{M}$ is a paracompact Hausdorff topological space, we have
\begin{equation}
\Gamma(K,\, \pi_{V*}\pboexp ) =
\underset{\Omega \supset K}{\varinjlim}\,\,
\Gamma(\Omega,\,\pi_{V*}\pboexp),
\end{equation}
where the limit is taken with respect to all open subsets
in $\partial M$ containing $K$.
Therefore every element of 
$\Gamma(K,\, \pi_{V*}\pboexp)$ 
can be first extended to an open neighborhood $\Omega$ of $K $ 
in $\partial M$.
We may assume that $L:=\partial M \setminus \Omega$ 
is a finite union of closed subsets 
$K_{i}\subset \partial M$ where each $K_{i} \subset \partial M$ is of product type in $\overline{M}$.
Consider the  exact sequence of  cohomology groups
\begin{equation}
\Gamma(\partial M,\, \pi_{V*}\pboexp)
\rightarrow
\Gamma(\Omega,\, \pi_{V*}\pboexp) 
\rightarrow
\mathrm{H}^{1}_{L}(\partial M,\, \pi_{V*}\pboexp).
\end{equation}
Then, by Proposition \ref{prop:flabby-dim-partial_R} and Proposition \ref{cover},
we have
$$
\mathrm{H}^{1}_{L}(\partial M,\, \pi_{V*}\pboexp)
=
\mathrm{H}^{n+1}_{L \times V}(\rcsp \times V,\,\hexpo)=0,
$$
from which surjectivity follows.  This completes the proof.
\end{proof}
Recall that  
$j : N \hookrightarrow \overline{N}$ 
denotes the natural embedding.
\begin{theorem}
Let $V$ be a Stein open subset in $T$, and let $\pi_V: \rcsp \times V
\to \rcsp$ denote the canonical projection.
We have
$$
\mathbf{R}\pi_{V*}\boexp \simeq \pi_{V*}\boexp
$$
and, for any closed subset $K \subset \overline{M}$,
the restriction morphism
\begin{equation}{\label{eq:sur_bo}}
\Gamma\left(\overline{M},\, \pi_{V*}\boexp\right) 
\rightarrow 
\Gamma\left(K,\, \pi_{V*}\boexp\right)
\end{equation}
is surjective, 
that is, the sheaf $\pi_{V*}\boexp$ is soft.
\end{theorem}
\begin{proof}
We have the distinguished triangle
\begin{equation}
\mathbf{R}\pi_{V*}\pboexp \to
\mathbf{R}\pi_{V*}\boexp \to
\mathbf{R}\pi_{V*}\mathbf{R}j_{*}\mathcal{B}\mathcal{O}_N \overset{+1}{\rightarrow}.
\end{equation}
It follows from Proposition \ref{prop:flabby-dim-partial_R} and 
$\mathbf{R}\pi_{V*}\mathbf{R}j_{*}\mathcal{B}\mathcal{O}_N =
(\pi_{V}\circ j)_{*}\mathcal{B}\mathcal{O}_N$ that the complex
$\mathbf{R}\pi_{V*}\boexp$ is concentrated in degree $0$ and we have the exact
sequence:
\begin{equation}{\label{eq:exact-3-sheaves}}
0 \to \pi_{V*}\pboexp \to \pi_{V*}\boexp \to
(\pi_{V} \circ j)_{*}\mathcal{B}\mathcal{O}_N \to 0.
\end{equation}
Then, since
$\pi_{V*}\pboexp$ and 
$(\pi_{V} \circ j)_{*}\mathcal{B}\mathcal{O}_N$ are soft, 
the $\pi_{V*}\boexp$ becomes soft.
This completes the proof.
\end{proof}
\begin{remark}
{\normalfont
Surjectivity of \eqref{eq:sur_bo} is equivalent to that of the restriction morphism
$$
\Gamma\left(\overline{M} \times V,\, \boexp\right) 
\rightarrow 
\underset{\Omega \supset\supset K}{\varinjlim}\,\,\Gamma\left(\Omega \times V,\, \boexp\right),
$$
where $\Omega$ ranges through open subsets in $\rcsp$ containing $K$.
}
\end{remark}
\begin{corollary}
The sheaf $\bexpm$ of Laplace hyperfunctions on $\overline{M}$ is soft.
\end{corollary}
By \eqref{eq:exact-3-sheaves} and softness of the sheaf 
$\pi_{V*}\pboexp$,
we also have the following theorem.
\begin{theorem}
For any open subset $\Omega \subset \overline{M}$ and
a Stein open subset $V$ in $T$, 
the restriction morphism
$\boexp(\Omega \times V) \rightarrow 
\mathcal{B}\mathcal{O}_{N}( (\Omega \times V) \cap N )$ 
is surjective.
\end{theorem}
\begin{corollary}
For any open subset $\Omega \subset \overline{M}$,
the restriction morphism
$\bexpm(\Omega) \rightarrow 
\mathcal{B}_{M}(\Omega \cap M)$ 
is surjective.
\end{corollary}

\end{document}